\documentclass[11pt,reqno]{amsart}

\usepackage[T1]{fontenc}
\usepackage{amsmath,amssymb,amsthm,mathtools}
\usepackage{booktabs}
\usepackage{graphicx}
\usepackage{float}
\usepackage{microtype}
\usepackage{enumitem}
\usepackage[colorlinks=true,linkcolor=blue,citecolor=blue,urlcolor=blue]{hyperref}
\usepackage[capitalize,noabbrev]{cleveref}
\usepackage{geometry}
\graphicspath{{figures/}}
\setlist{itemsep=2pt,topsep=4pt}

\newtheorem{theorem}{Theorem}[section]
\newtheorem{proposition}[theorem]{Proposition}
\newtheorem{corollary}[theorem]{Corollary}
\newtheorem{lemma}[theorem]{Lemma}
\newtheorem{definition}[theorem]{Definition}
\theoremstyle{remark}

\crefname{theorem}{Theorem}{Theorems}
\crefname{proposition}{Proposition}{Propositions}
\crefname{corollary}{Corollary}{Corollaries}
\crefname{lemma}{Lemma}{Lemmas}
\crefname{definition}{Definition}{Definitions}
\crefname{remark}{Remark}{Remarks}
\crefname{example}{Example}{Examples}

\newcommand{\C}{\mathbb C}
\newcommand{\R}{\mathbb R}
\newcommand{\Q}{\mathbb Q}
\newcommand{\Z}{\mathbb Z}
\newcommand{\PP}{\mathbb P}
\newcommand{\tr}{\operatorname{tr}}

\newcommand{\den}{\operatorname{den}}

\newcommand{\diag}{\operatorname{diag}}
\newcommand{\Log}{\operatorname{Log}}

\newcommand{\cA}{\mathcal A}

\newcommand{\cH}{\mathcal H}
\newcommand{\cR}{\mathcal R}
\newcommand{\cX}{\mathcal X}
\newcommand{\dd}{\,\mathrm d}

\title[Two-Channel Power-Law Spectral Curves]
{Algebraic Spectral Curves of Two-Channel Operator Pencils with Power-Law Response}

\author{Kejun Liu}
\address{State Key Laboratory of Bioinspired Interface Material Science, Soochow University, Suzhou 215123, China}
\email{kjliu@suda.edu.cn}
\date{July 31, 2026}


\hypersetup{
  pdftitle={Algebraic Spectral Curves of Two-Channel Operator Pencils with Power-Law Response},
  pdfauthor={Kejun Liu},
  pdfsubject={Joint determinant curves of two-channel pencils with causal power-law response},
  pdfkeywords={operator pencils, power-law response, hyperelliptic curves, causal memory, conformal response}
}

\begin{document}

\setcounter{tocdepth}{1}

\begin{abstract}
A finite matrix can acquire a nontrivial algebraic spectral curve when its
entries contain a multivalued response.  We study this mechanism for a
two-channel operator pencil with a scalar power law $z^\beta$ and
$J$-self-adjoint coefficient matrices.  If $\beta=r/n$ is in lowest terms,
then, away from an explicit degenerate locus, the normalized determinant
curve is a degree-two cover of genus $n-1$.  It is rational for $n=1$,
elliptic for $n=2$, and hyperelliptic for $n\geq3$.  Within this fixed
two-channel class, the generic genus is fixed by the denominator of the
response exponent and is independent of its numerator.  The proof follows
from a birational reduction to a classical hyperelliptic normal form with a
cyclic automorphism.  The operator-specific datum is finer: the numerator
appears as the cyclic character of the distinguished coupling function and
controls its divisor and raw boundary singularity.  Here $\lambda$ varies as a
complexified coupling coordinate: the curve is the joint $(z,\lambda)$
determinant locus, not the frequency spectrum of one operator at fixed
physical coupling.

For an irrational exponent the local monodromy has infinite order, so no
finite-sheeted response cover can make the power meromorphic.  On a fixed
logarithmic branch and at fixed coupling, however, rational approximants
converge uniformly on compact sets, and isolated spectral zeros are stable.
Negative exponents in $(-1,0)$ arise from a causal Volterra kernel and as
locally uniform limits of cutoff continuum-bath Schur complements.  This
separates the finite number of probe channels from the nonlocal memory that
produces the branching.

As an application, the exact leading infrared monomial model associated with
a zero-temperature conformal response of dimension $h_{\rm probe}$ has, for
rational $2h_{\rm probe}$ and generic probe coefficients, genus
$\den(2h_{\rm probe})-1$.  For the fundamental fermion in the leading
large-$N$, strong-coupling, zero-temperature saddle of the even-$q>2$
Sachdev--Ye--Kitaev model, this becomes $q/2-1$.  This is a statement about
the joint determinant family of the monomial infrared model, not the global
curve of a completed response or a fixed-coupling physical spectrum.  At
denominator two the response fixes genus one but not the elliptic modulus.
\end{abstract}

\maketitle
\enlargethispage{2pt}
\tableofcontents

\section{Introduction}\label{sec:introduction}

Consider a two-channel system whose frequency-domain response contains a
square root.  There are only two internal channels, but analytic continuation
around the origin exchanges two response sheets.  For the elementary pencil
used below, the zero set of the determinant is not a rational two-level
dispersion relation.  Its smooth projective normalization is an elliptic
curve.  This is the observation from which the paper starts.

Nonlinear eigenvalue problems with algebraic matrix entries, including
square-root examples, are well established \cite{GuttelTisseur2017}, as are
self-adjoint pencils and two-parameter determinant functions in Krein
theory \cite{KollarMiller2014}.  Determinant loci are also standardly studied
as spectral curves, including singular and reducible cases
\cite{Izosimov2015}.  For a generic size-$N$, degree-$n$ matrix-polynomial
spectral curve, Vivolo's formula gives genus
$(N-1)(Nn-2)/2$, together with corrections for prescribed singularities
\cite{Vivolo2000}.  The two-channel value below is the $N=2$ member of
that classical count.  More specifically, hyperelliptic curves with cyclic
automorphisms have classical normal forms $Y^2=F(X^n)$.  Shaska's normal
decomposition already contains $Y^2=X^{2n}+aX^n+1$ with genus $n-1$
\cite{Shaska2003}; related automorphism classifications retain the broader
form $y^2=x^kh(x^n)$ \cite{GutierrezShaska2005,MuellerPink2022}.  Neither this
bare curve family nor its genus count is claimed as new here.  The point is
its exact emergence from a two-channel causal response pencil and the
resulting distinguished coupling coordinate, whose cyclic character, divisor,
and boundary singularity retain the exponent numerator.

It is natural at first to ask whether the elliptic curve is a peculiarity of
the half-order exponent.  Replacing that response by $z^{r/n}$ produces, for
a generic pencil within the same two-channel class, a curve of genus
\begin{equation}\label{eq:intro-genus}
 g=n-1,
\end{equation}
where $r/n$ is reduced.  The numerator changes how the spectral and coupling
coordinates are related in the original pencil, but disappears from the
birational model.  The denominator records the order of the response monodromy and
survives as the number of branch points.  The statement is specific to the
two-channel determinant: one channel gives a rational graph, while two
channels supply the quadratic extension used below.  Response monodromy and
channel dimension are separate pieces of the geometry.

The matrix family is explicit.  Put
\begin{equation}\label{eq:intro-pencil}
 J=\diag(1,-1),\qquad H_0=EJ,\qquad
 V=\begin{pmatrix}a&b\\-b&d\end{pmatrix},
\end{equation}
with real parameters and $E\ne0$, and consider
\begin{equation}\label{eq:intro-general-pencil}
 \cA_\beta(z;\lambda)
 =zI_2-EJ-\lambda z^\beta V.
\end{equation}
The matrices $H_0$ and $V$ are $J$-self-adjoint.  For a complex spectral
parameter the correct statement is a Schwarz symmetry, not pointwise
$J$-self-adjointness; this distinction is made precise in
\cref{sec:causal-response}.  The scalar $z^\beta$ supplies all the
multivaluedness.

Throughout the paper, $\lambda$ is the complexified coupling coordinate of
the pencil.  A physical problem may instead fix a real value $\lambda_0$ and
study the zeros in $z$ alone.  Such a fixed-coupling fiber is a zero divisor,
not the compact curve whose genus is computed here.  The determinant spectral
curve is the joint complexified locus in $(z,\lambda)$, after passage to the
response cover and normalization.

For $\beta=r/n\in\Q$ in lowest terms, set $z=u^n$.  The determinant equation
is birational, on the active locus $u\ne0$, to
\begin{equation}\label{eq:intro-curve}
 y^2=(\tau u^n+E\delta)^2
      -4\Delta_V(u^{2n}-E^2),
\end{equation}
where
\begin{equation}\label{eq:intro-invariants}
 \tau=a+d,\qquad \delta=a-d,\qquad
 \Delta_V=ad+b^2.
\end{equation}
On the explicit generic locus, the right-hand side has degree $2n$ and
$2n$ simple roots; the double cover of $\PP^1_u$ is unramified at infinity.
Riemann--Hurwitz then gives \eqref{eq:intro-genus}.  Degenerate coefficients
do not require a separate geometric guess: the genus is read from the
odd-multiplicity square-class part of the branch polynomial, while root
multiplicities retain the singularity type.

The first main result, \cref{thm:denominator-genus}, is therefore a genus
formula and a birational invariance statement.  All reduced exponents with
the same denominator give the same determinant function field after a
change of the $\lambda$ coordinate.  The second main result,
\cref{thm:irrational-monodromy}, identifies the sharp obstruction beyond
rational powers.  If $\beta$ is irrational, continuation around the origin
has infinite-order monodromy.  No finite response cover can turn $z^\beta$
into a meromorphic function.  This assertion retains the original spectral
and coupling coordinates; a transcendental coordinate change can hide the
response in an abstract algebraic conic.  The obstruction coexists with a
useful local fact: on a fixed logarithmic branch, rational approximants
converge on compact sets and preserve isolated spectral zeros at fixed
coupling.

The power law is not inserted only to manufacture curves.  For
$0<\alpha<1$, the causal kernel
\begin{equation}\label{eq:intro-memory-kernel}
 k_\alpha(t)=\frac{\theta(t)t^{-\alpha}}{\Gamma(1-\alpha)}
\end{equation}
has Laplace transform $z^{\alpha-1}$.  The same function is the Stieltjes
transform of the positive continuum measure
\begin{equation}\label{eq:intro-stieltjes}
 \frac{\sin(\pi\alpha)}{\pi}\omega^{\alpha-1}\dd\omega.
\end{equation}
It is thus obtained from a Volterra memory equation and as the locally
uniform limit of cutoff continuum-bath Schur complements.  The reduced
system is finite-dimensional; the continuum realization is not.  This is the
elementary reason a small matrix can retain a high-order branch structure.

A conformal response gives a conditional application of the genus formula.
At zero temperature, a retarded two-point function of
dimension $h_{\rm probe}$ may have the leading low-frequency form
\begin{equation}\label{eq:intro-conformal}
 G_R(\omega)\sim c_h[-i(\omega+i0)]^{2h_{\rm probe}-1}.
\end{equation}
Replace this asymptotic expression by its exact leading infrared monomial
model.  For a generic two-channel probe and rational $2h_{\rm probe}$, the
joint determinant curve of that model has
\begin{equation}\label{eq:intro-ads-genus}
 g_{\mathrm{probe}}=\den(2h_{\rm probe})-1.
\end{equation}
This does not assign a global genus to a response known only asymptotically.
It does not identify the finite-temperature Green function with a power law,
and it does not determine the full spectrum of a black hole or of the bath.
For the fundamental fermion of the interacting even-$q$ SYK conformal saddle,
$\Delta_\psi=1/q$, and \eqref{eq:intro-ads-genus} becomes $q/2-1$.  The
familiar $q=4$ exponent gives genus one for the leading monomial probe model,
but does not select any particular point of elliptic moduli.

The half-order layer also displays coefficient-dependent moduli.  The
binary-quartic $j$-invariant varies with the matrix entries, while the
specialization $a=b=d=E=1$ has the model
\begin{equation}\label{eq:intro-canonical-elliptic}
 Y^2=X^3+8X.
\end{equation}
The curve is a property of that parameter choice, not of the exponent alone.
The resolvents of the same pencil also admit a transparent inverse problem:
one may prescribe their trace and determinant and hence a quadratic reverse
characteristic polynomial.  When the prescribed coefficients are written
as $(a_p,p)$, the resulting identity has the form of an inverse local Euler
polynomial.  Both coefficients were supplied to the construction.  The
result is local spectral interpolation, not an arithmetic selection rule.

The paper follows the calculation rather than a taxonomy.  A concrete
elliptic example is developed in \cref{sec:two-channel}; the causal origin
and the operator class are fixed in \cref{sec:causal-response}.  Rational
exponents, the genus theorem, and its degenerate loci occupy
\cref{sec:rational-exponents,sec:genus}.  Irrational powers are treated in
\cref{sec:irrational}.  The elliptic and near-$AdS_2$ consequences appear in
\cref{sec:elliptic,sec:ads2}; local interpolation is kept until
\cref{sec:interpolation}, where its information flow can be stated without
ambiguity.

No operator for a global $L$-function is constructed here.  Nor do we derive
Schwarzian dynamics, black-hole entropy, temperature dependence, or a
complete quasinormal-mode spectrum.  The narrow result near black-hole
physics is \eqref{eq:intro-ads-genus}: for an exact leading conformal
monomial and generic two-channel coefficients, the reduced denominator of a
probe scaling dimension determines the topology of the corresponding joint
infrared determinant family.  No microscopic two-channel coupling or its
genericity is derived from a black-hole model here.

\section{A Two-Channel Model}\label{sec:two-channel}

We first keep only the example that exposes the geometry.  Take
\begin{equation}\label{eq:canonical-data}
 E=1,\qquad
 V_c=\begin{pmatrix}1&1\\-1&1\end{pmatrix},
 \qquad \beta=-\frac12.
\end{equation}
Choose a square root $z=u^2$, so that $z^{-1/2}=u^{-1}$.  The pencil becomes
\begin{equation}\label{eq:canonical-pencil}
 \cA_c(u;\lambda)
 =u^2I_2-J-\frac{\lambda}{u}V_c,\qquad u\ne0.
\end{equation}
A direct determinant expansion gives
\begin{equation}\label{eq:canonical-P}
 P_c(u,\lambda):=u^2\det\cA_c(u;\lambda)
 =u^6-u^2-2\lambda u^3+2\lambda^2.
\end{equation}
The factor $u^2$ clears the response pole and is not part of the active
operator domain.  Accordingly, the determinant curve means the smooth
projective normalization of the closure of $P_c=0$ with $u\ne0$.

Equation \eqref{eq:canonical-P} is quadratic in $\lambda$.  The useful
coordinate is not its raw square root but
\begin{equation}\label{eq:canonical-v}
 v=\frac{2\lambda}{u}-u^2.
\end{equation}
Substitution yields
\begin{equation}\label{eq:canonical-quartic}
 v^2=2-u^4,
\end{equation}
and the inverse transformation is
\begin{equation}\label{eq:canonical-lambda-inverse}
 \lambda=\frac{u}{2}(u^2+v).
\end{equation}
Both maps are rational on dense open sets.  Thus the determinant curve and
the quartic have the same function field over $\Q$.

\begin{proposition}[Half-order determinant curve]\label{prop:half-order-elliptic}
The smooth projective normalization of \eqref{eq:canonical-P} is a curve of
genus one.  It is isomorphic over $\Q$ to the smooth projective model of
\eqref{eq:canonical-quartic}.
\end{proposition}

\begin{proof}
The birational equivalence follows from
\eqref{eq:canonical-v}--\eqref{eq:canonical-lambda-inverse}.  The projection
of \eqref{eq:canonical-quartic} to the $u$-line is a double cover.  Its finite
branch divisor consists of the four simple roots of $u^4=2$.  Since the
polynomial has even degree four, the two points above infinity are
unramified.  Riemann--Hurwitz gives
\[
 2g-2=2(-2)+4=0.
\]
Hence $g=1$.  The quartic has the rational point $(u,v)=(1,1)$, so after
choosing that point it is an elliptic curve over $\Q$.
\end{proof}

The four branch points are
\begin{equation}\label{eq:canonical-branch-points}
 2^{1/4},\quad -2^{1/4},\quad i2^{1/4},\quad -i2^{1/4}.
\end{equation}
\begin{figure}[H]
\centering
\includegraphics[width=\textwidth]{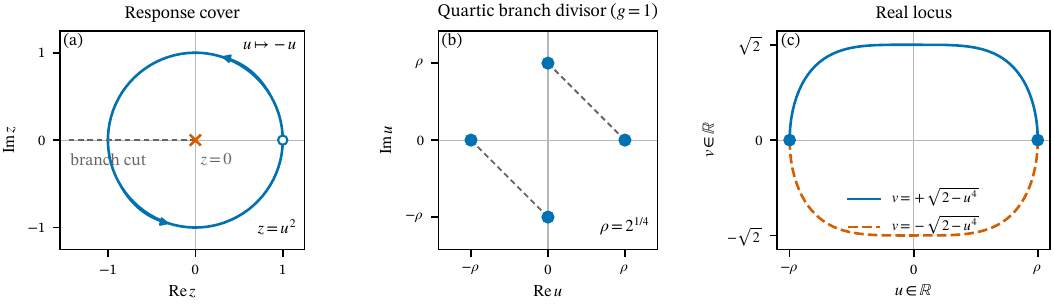}
\caption{The half-order example.  The response cover $z=u^2$ is rational
and has the deck transformation $u\mapsto-u$ (left).  The determinant
equation adds a double cover branched at the four roots of $u^4=2$ (middle),
whose real locus is shown at right.  Genus one belongs to the normalized
determinant curve, not to the response cover.}
\label{fig:half-order-geometry}
\end{figure}
Their cross-ratio is $-1$ up to the usual six cross-ratio transforms.  This
already suggests $j=1728$, although the full $j$-map will be derived from
the general quartic in \cref{sec:elliptic}.  Here the more important point
is structural: no entry of \eqref{eq:canonical-pencil} is a differential
operator, and there are only two channels.  The elliptic curve is generated
by the two-sheeted response and the additional quadratic determinant
equation.

The same calculation also warns against identifying three different covers.
The response cover $u\mapsto z=u^2$ has genus zero.  The matrix pencil is a
meromorphic family on that cover.  Its determinant equation adds a second
sheet in $\lambda$, and it is the normalization of this last curve that has
genus one.  We retain this distinction throughout.

The exponent $-1/2$ could look specially tuned.  To find out which part of
the answer is special, one should not begin by changing the matrix size.  It
is enough to replace the response by $z^{r/n}$.  Before doing so, we derive
where such powers occur in a causal problem.

\section{Causal Power-Law Response}\label{sec:causal-response}

\subsection{Volterra Realization}

Causal convolution equations and their fractional kernels have a standard
Laplace--Stieltjes formulation \cite{Kochubei2011}.  We record the particular
normalization needed for the determinant geometry.

Let $x(t)\in\C^N$ vanish for $t<0$ and consider the linear Volterra equation
\begin{equation}\label{eq:volterra}
 \dot x(t)-H_0x(t)
 -\lambda\int_0^t k(t-s)Vx(s)\dd s=f(t),
 \qquad x(0)=0.
\end{equation}
For $\Re z$ sufficiently large, its Laplace transform is
\begin{equation}\label{eq:volterra-laplace}
 \bigl[zI-H_0-\lambda\widehat k(z)V\bigr]\widehat x(z)
 =\widehat f(z).
\end{equation}
Thus a frequency-dependent matrix pencil is the exact reduced operator of a
causal memory equation.

Our transform conventions are fixed here.  For a causal distribution $k$ of
exponential type,
\begin{equation}\label{eq:transform-conventions}
 \widehat k(z)=\int_0^\infty e^{-zt}k(t)\dd t,
 \qquad
 \widehat k^R(\omega)
 =\lim_{\eta\downarrow0}\widehat k(\eta-i\omega),
\end{equation}
whenever the boundary value exists.  Thus $z=\eta-i\omega$ and our Fourier
sign is $e^{+i\omega t}$.  The advanced boundary value is obtained from
$z=\eta+i\omega$ for a real kernel.

For $0<\alpha<1$, set
\begin{equation}\label{eq:k-alpha}
 k_\alpha(t)=\frac{\theta(t)t^{-\alpha}}{\Gamma(1-\alpha)}.
\end{equation}
The singularity at the origin is locally integrable, and for $\Re z>0$
\begin{equation}\label{eq:laplace-power}
 \widehat k_\alpha(z)
 =\frac1{\Gamma(1-\alpha)}
   \int_0^\infty e^{-zt}t^{-\alpha}\dd t
 =z^{\alpha-1},
\end{equation}
with the principal logarithm.  Therefore the range
\begin{equation}\label{eq:negative-beta-range}
 -1<\beta<0,\qquad \beta=\alpha-1,
\end{equation}
is realized by an ordinary locally integrable causal kernel.  Positive
noninteger powers can be treated as fractional derivatives or as subtracted
self-energies, but they do not follow from \eqref{eq:k-alpha} without this
additional interpretation.  The algebraic results below allow arbitrary
real $\beta$; the direct long-memory realization is \eqref{eq:negative-beta-range}.

\subsection{Continuum-Bath Realization}

The same response has a Stieltjes representation.  For $0<\alpha<1$ and
$z\in\C\setminus(-\infty,0]$,
\begin{equation}\label{eq:stieltjes-power}
 \int_0^\infty\frac{\omega^{\alpha-1}}{z+\omega}\dd\omega
 =\frac{\pi}{\sin(\pi\alpha)}z^{\alpha-1}.
\end{equation}
Indeed, the substitution $\omega=zt$ proves the formula first for positive
real $z$ by the Euler beta integral; analytic continuation gives the stated
slit-plane identity.  Convergence at zero requires $\alpha>0$, while
convergence at infinity requires $\alpha<1$.

Define the positive measure
\begin{equation}\label{eq:bath-measure}
 \dd\mu_\alpha(\omega)
 =\frac{\sin(\pi\alpha)}{\pi}
  \omega^{\alpha-1}\dd\omega.
\end{equation}
Then
\begin{equation}\label{eq:bath-m}
 m_\alpha(z):=\int_0^\infty\frac{\dd\mu_\alpha(\omega)}{z+\omega}
 =z^{\alpha-1}.
\end{equation}
This identity is the continuum-bath origin of the branch cut.

With \eqref{eq:transform-conventions}, its retarded boundary value is
\begin{equation}\label{eq:retarded-power-boundary}
 m_\alpha^R(\omega)
 =[-i(\omega+i0)]^{\alpha-1}
 =|\omega|^{\alpha-1}
 \begin{cases}
 e^{i\pi(1-\alpha)/2},&\omega>0,\\
 e^{-i\pi(1-\alpha)/2},&\omega<0.
 \end{cases}
\end{equation}
The advanced value is its complex conjugate.  Some Green-function conventions
include an additional overall factor $-i$; this changes the displayed phase
and the corresponding spectral-density convention, but not the response
exponent or its monodromy.

For completeness, every cutoff response is a literal Schur complement.  This
is a concrete instance of the Feshbach--Schur reduction, whose general
isospectral form produces finite-dimensional energy-dependent effective
operators \cite{DussonSigalStamm2021}.
Let $D_R$ be multiplication by $\omega$ on
$L^2([R^{-1},R],\mu_\alpha)\otimes\C^N$, and let
$B_Rx(\omega)=x$ and
\begin{equation}\label{eq:C-R}
 C_Rg=V\int_{R^{-1}}^R g(\omega)\dd\mu_\alpha(\omega).
\end{equation}
Both maps are bounded because the truncated measure is finite.  The Schur
complement of the lower-right block in
\begin{equation}\label{eq:block-pencil}
 \begin{pmatrix}
 zI-H_0&-\lambda C_R\\
 -B_R&z+D_R
 \end{pmatrix}
\end{equation}
is
\begin{equation}\label{eq:truncated-schur}
 zI-H_0-\lambda m_{\alpha,R}(z)V,
 \qquad
 m_{\alpha,R}(z)=
 \int_{R^{-1}}^R\frac{\dd\mu_\alpha(\omega)}{z+\omega}.
\end{equation}
The maps $B_R$ and $C_R$ give an algebraic left--right realization; for a
generic indefinite $V$ they are not Hilbert-space adjoints.  There is,
however, a reciprocal finite-cutoff realization in a Krein bath.  If
$V^{[*]}=V$, then $JV$ is Hermitian and admits a factorization
\begin{equation}\label{eq:krein-factorization}
 JV=C^*J_bC,
 \qquad
 V=C^{[*]}C,
 \qquad C^{[*]}=JC^*J_b,
\end{equation}
where $J_b$ is the signature of the nonzero spectral part of $JV$.  On
$L^2([R^{-1},R],\mu_\alpha)\otimes\mathcal F$, let $D_R$ again denote
multiplication by $\omega$ and put $W_Rx(\omega)=Cx$.  For real
$\lambda\geq0$, the block generator
\begin{equation}\label{eq:krein-bath-generator}
 \mathbb L_R=
 \begin{pmatrix}
  H_0&\sqrt\lambda\,W_R^{[*]}\\
  \sqrt\lambda\,W_R&-D_R
 \end{pmatrix}
\end{equation}
is self-adjoint for the product signature $J\oplus J_b$, and the upper Schur
complement of $zI-\mathbb L_R$ is again \eqref{eq:truncated-schur}.  Thus the
finite-cutoff response has a reciprocal realization once an indefinite bath
signature is allowed.  This statement asserts neither Hilbert-space
passivity nor self-adjointness of the parent block for complex $\lambda$.

On every compact subset of the slit plane,
$m_{\alpha,R}\to m_\alpha$ locally uniformly.  The limiting self-energy is
therefore the exact Stieltjes function \eqref{eq:bath-m}; the cutoff argument
also makes clear where the infinite bath has entered.  The uncut measure has
infinite total mass, so its constant coupling vector is not an element of the
bath Hilbert space.  What is proved here is the locally uniform limit of the
cutoff Schur complements.  An uncut parent-operator realization would require
a separate singular or form-coupling construction, including its domain and
closedness theory; no such construction is asserted here.  We also do not
assert that the block matrix \eqref{eq:block-pencil} is
Hilbert-self-adjoint for an arbitrary indefinite $V$.  Its role is the exact
algebraic elimination of truncated continuum channels and their locally
uniform Schur-complement limit.

\subsection{Operator Class}

We now fix the terminology used in the rest of the paper.  Let
$(\cX,J)$ be a finite-dimensional Krein space, so $J=J^*=J^{-1}$, and write
\begin{equation}\label{eq:j-adjoint}
 T^{[*]}=JT^*J.
\end{equation}

\begin{definition}[Riemann-surface causal operator pencil]\label{def:rsco}
Let $k$ be a real causal distribution of exponential type whose Laplace
transform is holomorphic in a right half-plane.  Let
$\pi_\Sigma:\cR_\Sigma\to\Omega_z$ be a specified analytic-continuation
surface, possibly infinite-sheeted, equipped with an antiholomorphic
involution $\iota$ above complex conjugation:
\begin{equation}\label{eq:real-structure}
 \pi_\Sigma(\iota p)=\overline{\pi_\Sigma(p)}.
\end{equation}
Let $\widetilde\Sigma$ be the single-valued meromorphic continuation of the
Laplace transform to $\cR_\Sigma$.  Suppose:
\begin{enumerate}[label=\textnormal{(\roman*)}]
\item $H_0^{[*]}=H_0$ and $V^{[*]}=V$ on $(\cX,J)$;
\item $\widetilde\Sigma$ agrees with the Laplace transform on the physical
sheet;
\item $\widetilde\Sigma(\iota p)
=\overline{\widetilde\Sigma(p)}$ whenever both sides are defined.
\end{enumerate}
The meromorphic family, with $\lambda\in\C$ as a complexified coupling
coordinate,
\begin{equation}\label{eq:rsco-definition}
 \cA(p;\lambda)=\pi_\Sigma(p)I-H_0
 -\lambda\widetilde\Sigma(p)V
\end{equation}
will be called a \emph{Riemann-surface causal operator pencil}, or an
\emph{RSCO pencil}.
\end{definition}

The word ``causal'' refers to the support of the time-domain kernel.  ``Riemann
surface'' refers to the single-valued analytic continuation of its scalar
response.  ``Operator pencil'' refers to the finite-channel family
\eqref{eq:rsco-definition}.  The causal kernel, scalar response, response
surface, operator pencil, and determinant curve are distinct objects.  In
particular, finite-dimensional channel space does not make the Volterra
dynamics local in time.

We use ``determinant spectral curve'' in the established matrix-pencil sense
\cite{Izosimov2015}, with the additional normalization and pole saturation
needed for the present multivalued response.

\begin{definition}[Determinant spectral curve]\label{def:spectral-curve}
For a finite algebraic response cover, the determinant spectral curve is the
smooth projective normalization of the reduced closure of
\begin{equation}\label{eq:determinant-locus}
 \{(p,\lambda)\in\cR_\Sigma\times\C:
   \det\cA(p;\lambda)=0\},
\end{equation}
after components introduced only by clearing poles have been removed, or
equivalently after saturation by the pole divisor.  If the locus is
reducible, its normalized components are recorded separately.
\end{definition}

The response cover and the determinant curve need not have the same genus.
In \cref{sec:two-channel}, the cover $u\mapsto u^2$ is rational while the
determinant curve is elliptic.

Nor is the determinant curve the spectrum of a fixed-coupling operator.  A
choice $\lambda=\lambda_0$ cuts \eqref{eq:determinant-locus} by a fiber and
leaves a zero divisor in the spectral variable.  The genus below belongs to
the joint complexified $(p,\lambda)$ locus.  Varying $\lambda$ is therefore
part of the definition of the curve, even when a real slice later interprets
$\lambda$ as a physical coupling.

Finally, coefficientwise $J$-self-adjointness must not be confused with a
pointwise assertion at complex parameters.  With the involution above,
\begin{equation}\label{eq:schwarz-symmetry}
 \cA(p;\lambda)^{[*]}
 =\cA(\iota p;\overline\lambda).
\end{equation}
Only at fixed points of the real structure and for real $\lambda$ is the
matrix itself $J$-self-adjoint.  The later algebraic geometry is over $\C$;
the symmetry is retained as \eqref{eq:schwarz-symmetry} rather than silently
extended beyond its domain.  On a real-frequency cut it pairs retarded and
advanced boundary values, $\cA^R(\omega)^{[*]}=\cA^A(\omega)$.

\section{Rational Response Exponents}\label{sec:rational-exponents}

Let $\beta=r/n$ with $r\in\Z$, $n\ge1$, and $\gcd(r,n)=1$.  The
minimal cyclic cover on which the response is single-valued is
\begin{equation}\label{eq:rational-cover}
 z=u^n,\qquad z^\beta=u^r.
\end{equation}
The point $u=0$ may be a zero or a pole of the lifted response.  We call
$u\ne0$ the active locus; all birational statements below are made there
before compactification.

We now restrict the RSCO definition in \cref{def:rsco} to the explicit
two-channel family \eqref{eq:intro-general-pencil}; no genus claim is made for
the general operator class.

Take the two-channel data \eqref{eq:intro-pencil} and retain the invariants
\begin{equation}\label{eq:invariants-repeat}
 \tau=a+d,\qquad \delta=a-d,\qquad \Delta_V=ad+b^2.
\end{equation}

\begin{lemma}[Raw determinant]\label{lem:raw-determinant}
On the cover \eqref{eq:rational-cover},
\begin{align}
 F_{r,n}(u,\lambda)
 &:=
 \det\cA_{r/n}(u^n;\lambda)\notag\\
 &=u^{2n}-E^2
 -\lambda u^r(\tau u^n+E\delta)
 +\Delta_V\lambda^2u^{2r}.
 \label{eq:raw-determinant}
\end{align}
\end{lemma}

\begin{proof}
The lifted matrix is
\[
 \begin{pmatrix}
 u^n-E-a\lambda u^r&-b\lambda u^r\\
 b\lambda u^r&u^n+E-d\lambda u^r
 \end{pmatrix}.
\]
Expanding its determinant and collecting the terms linear in $\lambda$
gives
\[
 d(u^n-E)+a(u^n+E)=\tau u^n+E\delta,
\]
while the quadratic coefficient is $ad+b^2=\Delta_V$.
\end{proof}

The determinant is quadratic in $\lambda$ for every exponent.  Its
discriminant contains the square $u^{2r}$.  Removing this coordinate square
is exactly the step that exposes the curve.

It is useful to separate that step from completing the square.  On the active
locus set
\begin{equation}\label{eq:q-coordinate}
 q=\lambda u^r.
\end{equation}
Then \eqref{eq:raw-determinant} becomes
\begin{equation}\label{eq:q-conic}
 \Delta_Vq^2-(\tau u^n+E\delta)q+u^{2n}-E^2=0.
\end{equation}
This equation already shows why the numerator disappears from the function
field, while the original coordinate $\lambda=u^{-r}q$ still remembers it.

\begin{proposition}[Birational reduction]\label{prop:birational-reduction}
Assume $\Delta_V\ne0$.  On $u\ne0$, the determinant curve
$F_{r,n}=0$ is birational to
\begin{equation}\label{eq:Qn-curve}
 \cH_n:\qquad y^2=Q_n(u),
\end{equation}
where
\begin{equation}\label{eq:Qn}
 Q_n(u)=(\tau u^n+E\delta)^2
        -4\Delta_V(u^{2n}-E^2).
\end{equation}
The maps are
\begin{align}
 y&=2\Delta_V\lambda u^r-(\tau u^n+E\delta),
 \label{eq:y-map}\\
 \lambda&=
 \frac{y+\tau u^n+E\delta}{2\Delta_Vu^r}.
 \label{eq:lambda-map}
\end{align}
\end{proposition}

\begin{proof}
Direct expansion, without use of the determinant equation, gives
\begin{equation}\label{eq:birational-identity}
 \bigl[2\Delta_V\lambda u^r-(\tau u^n+E\delta)\bigr]^2
 -Q_n(u)=4\Delta_VF_{r,n}(u,\lambda).
\end{equation}
Thus \eqref{eq:y-map} maps the determinant locus into \eqref{eq:Qn-curve}.
Equation \eqref{eq:lambda-map} is its rational inverse wherever
$\Delta_Vu^r\ne0$.
\end{proof}

The right-hand side of \eqref{eq:Qn} is
\begin{equation}\label{eq:Qn-ABC}
 Q_n(u)=Au^{2n}+Bu^n+C,
\end{equation}
with
\begin{equation}\label{eq:ABC}
 A=\tau^2-4\Delta_V,\qquad
 B=2E\tau\delta,\qquad
 C=E^2(\delta^2+4\Delta_V).
\end{equation}
Neither $r$ nor its sign appears.  This is stronger than a degree count.
When $AC\ne0$, scaling $u$ and $y$ puts \eqref{eq:Qn-ABC} in the classical
cyclic-automorphism form $Y^2=X^{2n}+aX^n+1$ \cite{Shaska2003}.  The new issue
for the pencil is therefore not the existence of that hyperelliptic normal
form, but which response data survive or disappear under the reduction.

\begin{corollary}[Numerator invariance]\label{cor:numerator-invariance}
Fix $n$, $E$, and $V$ with $\Delta_V\ne0$.  For any two integers $r_1,r_2$
coprime to $n$, suppose the common double-cover model is geometrically
integral.  Then the active determinant curves of the exponents $r_1/n$ and
$r_2/n$ have isomorphic function fields.  An isomorphism is obtained by using
the common coordinates $(u,y)$ of \eqref{eq:y-map}.
\end{corollary}

\begin{proposition}[Coupling character]\label{prop:coupling-character}
Let $\zeta_n=e^{2\pi i/n}$ and assume the common double-cover model in
\cref{cor:numerator-invariance} is geometrically integral.  The map
\begin{equation}\label{eq:cyclic-automorphism}
 \sigma(u,y)=(\zeta_nu,y)
\end{equation}
is an automorphism of its normalization.  For exponent $r/n$, the distinguished
coupling function obtained from \eqref{eq:lambda-map} satisfies
\begin{equation}\label{eq:coupling-character}
 \sigma^*\lambda_r=\zeta_n^{-r}\lambda_r.
\end{equation}
Equivalently, on the raw active determinant locus,
\begin{equation}\label{eq:raw-cyclic-action}
 (u,\lambda)\longmapsto(\zeta_nu,\zeta_n^{-r}\lambda).
\end{equation}
For $n>1$, coprimality makes this a primitive character of the cyclic response
group.
\end{proposition}

\begin{proof}
The polynomial $Q_n$ depends only on $u^n$, so
\eqref{eq:cyclic-automorphism} preserves the double-cover equation.  In the
inverse formula \eqref{eq:lambda-map}, its numerator is invariant and its
denominator acquires the factor $\zeta_n^r$, proving
\eqref{eq:coupling-character}.  The same substitution leaves every term of
$F_{r,n}$ invariant and gives \eqref{eq:raw-cyclic-action}.
\end{proof}

The numerator therefore survives as a character of the distinguished coupling
coordinate.  It also determines the pole or zero order of $\lambda_r$ at
$u=0$ and at infinity.  What it does not determine is the normalized
determinant function field.  This separation is the reason the genus theorem
depends on a reduced denominator alone.

The first cases already show the pattern.  For a single exact specialization,
take $E=1$ and $(a,b,d)=(2,1,3)$.  Then
\begin{equation}\label{eq:low-denominator-family}
 Q_n(u)=-3u^{2n}-10u^n+29.
\end{equation}
The quadratic in $u^n$ has two distinct nonzero roots.  Hence
\begin{align*}
 n=1:&\quad y^2=-3u^2-10u+29, &&g=0,\\
 n=2:&\quad y^2=-3u^4-10u^2+29, &&g=1,\\
 n=3:&\quad y^2=-3u^6-10u^3+29, &&g=2.
\end{align*}
These examples do not prove a general formula, since a degree-$2n$
polynomial can lose degree or acquire repeated roots.  They instead isolate
the question answered in the next section: which coefficient locus preserves
all $2n$ branch points?

\begin{figure}[H]
\centering
\includegraphics[width=\textwidth]{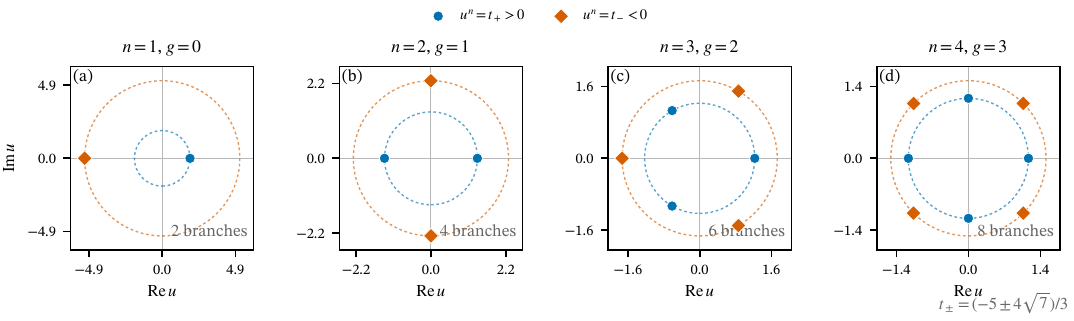}
\caption{Branch divisors for the fixed specialization
$Q_n=-3u^{2n}-10u^n+29$.  The two roots $t_\pm$ of the quadratic in $u^n$
each lift to an orbit of $n$ points.  The sequence $2,4,6,8$ of finite branch
points gives genera $0,1,2,3$ for $n=1,2,3,4$.}
\label{fig:branch-sequence}
\end{figure}

\section{Genus Formula and Degenerate Loci}\label{sec:genus}

We work over an algebraically closed field of characteristic zero when
discussing geometric genus.  The physical parameters may remain real; the
normalization and its genus are computed after complexification.  This
section separates three questions that a plane equation can obscure: whether
the branch polynomial is square-free, what the normalization is, and which
singularity occurs before normalization.

\subsection{Square-Free Locus}

Recall
\begin{equation}\label{eq:Qn-repeat}
 Q_n(u)=Au^{2n}+Bu^n+C,
\end{equation}
with $A,B,C$ given by \eqref{eq:ABC}.  Its discriminant as a polynomial in
$u$ is explicit.

\begin{lemma}[Polynomial discriminant]\label{lem:polynomial-discriminant}
For every $n\geq1$,
\begin{equation}\label{eq:Qn-discriminant}
 \operatorname{Disc}_u(Q_n)
 =n^{2n}(AC)^{n-1}(B^2-4AC)^n.
\end{equation}
Moreover,
\begin{equation}\label{eq:branch-discriminant}
 B^2-4AC=64E^2\Delta_Vb^2.
\end{equation}
For $n\geq2$, $Q_n$ is square-free of exact degree $2n$ if and only if
\begin{equation}\label{eq:generic-locus}
 E\Delta_Vb(\tau^2-4\Delta_V)(\delta^2+4\Delta_V)\ne0.
\end{equation}
For $n=1$, the exact condition is
$A(B^2-4AC)\ne0$; the stronger condition
\eqref{eq:generic-locus} is sufficient and will be used uniformly.
\end{lemma}

\begin{proof}
Write $f(t)=At^2+Bt+C$.  When $AC(B^2-4AC)\ne0$, its two roots
$t_1,t_2$ are distinct and nonzero, and
\[
 Q_n(u)=A(u^n-t_1)(u^n-t_2).
\]
Taking pairwise products of root differences, or equivalently taking the
resultant of $Q_n$ and $Q_n'$, gives \eqref{eq:Qn-discriminant}; the identity
extends polynomially to all coefficients.  Substitution from
\eqref{eq:ABC} gives
\begin{align*}
 B^2-4AC
 &=4E^2\left[
 \tau^2\delta^2
 -(\tau^2-4\Delta_V)(\delta^2+4\Delta_V)
 \right]\\
 &=64E^2\Delta_Vb^2.
\end{align*}
For $n\geq2$, the factors $A$, $C$, and $B^2-4AC$ in
\eqref{eq:Qn-discriminant} are all necessary.  When $n=1$, a zero constant
term can be a simple root, so $C\ne0$ is not necessary.
\end{proof}

\begin{theorem}[Denominator--genus formula for the pencil]\label{thm:denominator-genus}
Let $\beta=r/n$ be reduced, with $n\geq1$, and suppose
\eqref{eq:generic-locus} holds.  The smooth projective normalization of the
active determinant curve of \eqref{eq:intro-general-pencil} is a degree-two
cover of $\PP^1_u$ of genus
\begin{equation}\label{eq:main-genus}
 g=n-1.
\end{equation}
It is rational for $n=1$, elliptic for $n=2$, and hyperelliptic for
$n\geq3$.  For fixed $n$, its function field is independent of the numerator
$r$.
\end{theorem}

\begin{proof}
By \cref{prop:birational-reduction}, the determinant curve has the function
field of $y^2=Q_n(u)$.  By \cref{lem:polynomial-discriminant}, $Q_n$ has
degree $2n$ and $2n$ distinct finite roots.  The degree-two map to the
$u$-line is branched at these roots.  There are two points above infinity,
neither ramified.  Riemann--Hurwitz gives
\[
 2g-2=2(-2)+2n,
\]
which is \eqref{eq:main-genus}.  Geometric integrality and numerator
independence follow from \cref{prop:birational-reduction,cor:numerator-invariance}.
\end{proof}

The curve-theoretic genus in this theorem is the standard genus of the known
normal form $Y^2=X^{2n}+aX^n+1$ \cite{Shaska2003}.  The content specific to the
pencil is its determinant reduction, the identification of $n$ with the
reduced response denominator, and the numerator-sensitive decoration made
explicit in \cref{prop:coupling-character,thm:numerator-decoration}.

The theorem is a statement on a nonempty Zariski-open subset of coefficient
space.  The specialization \eqref{eq:low-denominator-family} proves
nonemptiness for every $n$.  It does not say that every pencil with response
denominator $n$ has genus $n-1$.

\subsection{Covering Maps}

The curve carries two natural maps, and their ramification data should not be
confused.  On the generic locus, the hyperelliptic projection
\begin{equation}\label{eq:pi-u}
 \pi_u:\cH_n\longrightarrow\PP^1_u
\end{equation}
has degree two and is unramified at its two points $P_\infty^\pm$ over
$u=\infty$.  In the chart $t=1/u$, $v=y/u^n$, its equation is
\begin{equation}\label{eq:infinity-chart}
 v^2=A+Bt^n+Ct^{2n},
\end{equation}
so these points are $v=\pm\sqrt A$.

The composed spectral map
\begin{equation}\label{eq:pi-z}
 \pi_z:\cH_n\longrightarrow\PP^1_z,
 \qquad z=u^n,
\end{equation}
has degree $2n$.  Each of $P_\infty^\pm$ has ramification index $n$, as does
each of the two points $P_0^\pm$ above $u=0$.  The $2n$ zeros of $Q_n$ each
have index two.  Thus the total ramification of $\pi_z$ is
\begin{equation}\label{eq:pi-z-total-ramification}
 2n+4(n-1)=6n-4,
\end{equation}
and Riemann--Hurwitz for the degree-$2n$ map again gives $g=n-1$.

The numerator remains visible in the original coupling coordinate.  On the
generic curve, the coordinate $q$ of \eqref{eq:q-coordinate} is finite and
nonzero at $P_0^\pm$ and has a pole of order $n$ at $P_\infty^\pm$.  Since
$\lambda=u^{-r}q$,
\begin{equation}\label{eq:lambda-divisor-orders}
 \operatorname{ord}_{P_0^\pm}(\lambda)=-r,
 \qquad
 \operatorname{ord}_{P_\infty^\pm}(\lambda)=r-n.
\end{equation}
Thus numerator invariance concerns the normalized function field and genus,
not the divisor of $\lambda$ or the raw plane embedding.

\subsection{Square-Class Reduction}

For a nonzero polynomial $Q$, write
\begin{equation}\label{eq:squareclass-factor}
 Q(u)=H(u)^2S(u),
\end{equation}
where $S$ retains precisely the irreducible factors of odd multiplicity and
is square-free.  We call $S$ the \emph{odd-multiplicity square-class part}.
It is not the polynomial radical: for example, if
$Q=(u^n-1)^2$, then $S=1$, not $u^n-1$.

\begin{proposition}[Genus from the square class]\label{prop:squarefree-genus}
Assume $\Delta_V\ne0$ and let $d=\deg S$ in
\eqref{eq:squareclass-factor}.  If $d\geq1$, the curve $y^2=Q(u)$ is
geometrically integral and its normalization has genus
\begin{equation}\label{eq:squarefree-genus}
 g=\left\lfloor\frac{d-1}{2}\right\rfloor.
\end{equation}
If $S$ is constant, the reduced curve splits over the algebraic closure into
two rational components.  The multiplicities removed in forming $S$ must
still be retained when classifying singularities of the plane model.
\end{proposition}

\begin{proof}
When $S$ is nonconstant, $Y=y/H(u)$ identifies the function field with
$Y^2=S(u)$.  If $d$ is even, the double cover has $d$ finite branch points
and is unramified above infinity.  If $d$ is odd, it has one further branch
point at infinity.  Riemann--Hurwitz gives
\eqref{eq:squarefree-genus}.  If $S$ is constant, the equation factors as
$(y-cH)(y+cH)=0$ after choosing a square root $c$ of the constant.
\end{proof}

\subsection{Degenerate Loci}

For $E\ne0$, the coefficient strata admit an exhaustive description.  First
suppose $\Delta_V=0$ and $V\ne0$.  Then \eqref{eq:q-conic} is genuinely linear
in $q$.  After common factors are removed, its main active component is the
rational graph
\begin{equation}\label{eq:delta-zero-param}
 \lambda=
 \frac{u^{2n}-E^2}{u^r(\tau u^n+E\delta)}.
\end{equation}
Common zeros of numerator and denominator can add vertical components; the
quadratic double-cover model is not valid on this locus.
If $V=0$, the determinant is instead $u^{2n}-E^2$, independent of $\lambda$,
and the active curve consists only of the vertical $\lambda$-lines over its
$2n$ roots.

Now suppose $\Delta_V\ne0$.  If $b=0$, the two channels decouple and
\begin{equation}\label{eq:bzero-square}
 Q_n(u)=(\delta u^n+E\tau)^2.
\end{equation}
The reduced curve has two rational components.  If $b\ne0$, then
$B^2-4AC\ne0$, and only the vanishing pattern of $A$ and $C$ remains:
\begin{enumerate}[label=\textnormal{(\alph*)}]
\item If $AC\ne0$, $Q_n$ is square-free of degree $2n$ and $g=n-1$.
\item If exactly one of $A,C$ vanishes, the odd-multiplicity square-class
part gives $g=\lfloor(n-1)/2\rfloor$.
\item If $A=C=0$, then $Q_n=Bu^n$.  For odd $n$ the normalization is one
rational component; for even $n$ the reduced curve splits into two rational
components.
\end{enumerate}

The middle case has two different plane models.  If $A=0$ and $BC\ne0$,
then $Q_n=Bu^n+C$ is square-free of degree $n$.  Infinity is ramified for
odd $n$ and splits into two unramified points for even $n$.  If $C=0$ and
$AB\ne0$, then
\begin{equation}\label{eq:Czero-model}
 Q_n=u^n(Au^n+B).
\end{equation}
At $u=0$ this has local type $A_{n-1}$ for $n\geq2$: it is a node only for
$n=2$, a cusp for $n=3$, and a higher singularity thereafter.  Removing the
even part of the multiplicity gives degree $n$ when $n$ is even and degree
$n+1$ when $n$ is odd, yielding the same genus
$\lfloor(n-1)/2\rfloor$.

For reference, \cref{tab:regimes} includes the irrational response beside
the rational strata.  The last row concerns the response cover, not an
additional algebraic degeneration.

\begin{table}[t]
\centering
\caption{Geometry of the two-channel power-law determinant.  Rows describe
generic points of the stated strata; intersections are resolved by
\cref{prop:squarefree-genus}.}
\label{tab:regimes}
\begin{tabular}{@{}p{0.22\textwidth}p{0.29\textwidth}p{0.37\textwidth}@{}}
\toprule
Parameter regime & Branch equation & Normalized geometry\\
\midrule
$E\Delta_VbAC\ne0$ & square-free degree $2n$ & degree-two cover, $g=n-1$\\
$\Delta_V=0$, $V\ne0$ & linear in $q$ & rational main active component\\
$V=0$ & independent of $\lambda$ & $2n$ vertical rational components\\
$b=0$, $\Delta_V\ne0$ & perfect square & two rational scalar components\\
$b\ne0$, exactly one of $A,C=0$ & square class of degree $n$ or $n+1$ &
$g=\lfloor(n-1)/2\rfloor$\\
$b\ne0$, $A=C=0$ & $Bu^n$ & one or two rational components\\
$\beta\notin\Q$ & infinite response monodromy & no compatible finite response cover\\
\bottomrule
\end{tabular}
\end{table}

\subsection{Boundary Singularities}

Even on the generic locus, the raw plane closure remembers the numerator.
Let $s=|r|>0$ and assume $C\ne0$.  For a negative numerator $r=-s$, clearing
the Laurent pole gives
\begin{equation}\label{eq:negative-r-closure}
 u^{2s}(u^{2n}-E^2)
 -\lambda u^s(\tau u^n+E\delta)
 +\Delta_V\lambda^2=0.
\end{equation}
Its discriminant in $\lambda$ is $u^{2s}Q_n(u)$.  The two branches through
$(u,\lambda)=(0,0)$ therefore differ by $u^s$ times a unit and have local
type $A_{2s-1}$.  It is a node only for $s=1$.

For a positive numerator $r=s$, use the reciprocal chart
$\mu=1/\lambda$.  The closure is
\begin{equation}\label{eq:positive-r-closure}
 (u^{2n}-E^2)\mu^2
 -u^s(\tau u^n+E\delta)\mu
 +\Delta_Vu^{2s}=0,
\end{equation}
with the same discriminant and singularity type at $(u,\mu)=(0,0)$.
Normalization separates these tangent branches.  This is why the numerator
changes the boundary embedding but not the geometric genus.

\begin{theorem}[Numerator-sensitive decoration]
\label{thm:numerator-decoration}
Assume the generic locus \eqref{eq:generic-locus} and let $r\ne0$.  For fixed
$n$, changing the reduced numerator leaves the normalized determinant curve
unchanged but changes its distinguished coupling data in three linked ways:
\begin{enumerate}[label=\textnormal{(\roman*)}]
\item $\lambda_r$ transforms by the primitive cyclic character
$\sigma^*\lambda_r=\zeta_n^{-r}\lambda_r$;
\item its orders at the two points over zero and infinity are
\begin{equation}\label{eq:decoration-orders}
 \operatorname{ord}_{P_0^\pm}(\lambda_r)=-r,
 \qquad
 \operatorname{ord}_{P_\infty^\pm}(\lambda_r)=r-n;
\end{equation}
\item the raw closure has an $A_{2|r|-1}$ boundary singularity over $u=0$,
in the $\lambda$ chart for $r<0$ and the reciprocal chart for $r>0$.
\end{enumerate}
Thus numerator invariance holds for the unmarked function field, not for the
curve equipped with its coupling function and raw boundary model.
\end{theorem}

\begin{proof}
Part (i) is \cref{prop:coupling-character}; part (ii) is
\eqref{eq:lambda-divisor-orders}.  The discriminants of
\eqref{eq:negative-r-closure} and \eqref{eq:positive-r-closure} are
$u^{2|r|}Q_n(u)$.  Since the generic locus gives $Q_n(0)\ne0$, the two smooth
local branches differ by $u^{|r|}$ times a unit, which is the
$A_{2|r|-1}$ type in part (iii).
\end{proof}

\subsection{Parameter Dependence}

The genus is locally constant off the discriminant locus.  On that locus the
geometric genus can drop, and the reduced curve can split.  The response
denominator fixes the generic value only after the two-channel class and a
square-free coefficient locus have been specified.
Matrix coefficients then determine branch collisions and, in the elliptic
case, the modulus.  This division of labor will be used in
\cref{sec:ads2}: a conformal dimension supplies a candidate denominator,
while the finite probe couplings decide whether the generic topology is
actually attained.

\section{Irrational Response Exponents}\label{sec:irrational}

The rational theory used a finite cyclic response cover.  For an irrational
exponent there is still a well-defined analytic response on the logarithmic
cover, but no finite cover can make that response meromorphic.

Let $\widetilde{\C^\times}=\C$ with coordinate $w$ and covering map
\begin{equation}\label{eq:log-cover}
 z=e^w.
\end{equation}
The lifted response is $e^{\beta w}$.  The deck transformation
$w\mapsto w+2\pi i$ acts by
\begin{equation}\label{eq:monodromy-multiplier}
 e^{\beta w}\longmapsto e^{2\pi i\beta}e^{\beta w}.
\end{equation}

\begin{theorem}[Infinite monodromy]\label{thm:irrational-monodromy}
If $\beta\in\R\setminus\Q$, the local monodromy of $z^\beta$ about $z=0$
has infinite order.  There is no finite-sheeted branched cover of a punctured
neighborhood of zero on which $z^\beta$ becomes a single-valued meromorphic
function.
\end{theorem}

\begin{proof}
By \eqref{eq:monodromy-multiplier}, $m$ turns about the origin multiply the
response by $e^{2\pi im\beta}$.  This is one only if $m\beta\in\Z$.  No
positive integer $m$ has this property when $\beta$ is irrational.

On a finite-sheeted branched cover, the monodromy permutation around a
puncture has finite order.  After some positive number $m$ of turns every
sheet returns to itself, and a single-valued meromorphic function must return
to its original value.  Applied to $z^\beta$, this would require
$e^{2\pi im\beta}=1$, a contradiction.
\end{proof}

The coordinate qualification is important.  If one makes the transcendental
change
\begin{equation}\label{eq:irrational-q-coordinate}
 q=\lambda z^\beta,
\end{equation}
then the raw determinant equation takes the algebraic form
\begin{equation}\label{eq:irrational-abstract-conic}
 \Delta_Vq^2-(\tau z+E\delta)q+z^2-E^2=0.
\end{equation}
Thus an abstract conic remains after the response has been absorbed into a
new coordinate.  What \cref{thm:irrational-monodromy} forbids is a finite
cover on which the original $z^\beta$, and hence the original coupling
coordinate $\lambda=z^{-\beta}q$ on a nontrivial branch, are both meromorphic.
It does not forbid an algebraic model obtained by forgetting that coordinate.
Degenerate pencils in which the response disappears are outside this
coordinate-compatible obstruction.

The conclusion concerns finite algebraization, not local approximation.
Choose a simply connected domain $\Omega\subset\C^\times$ and a holomorphic
branch $\Log_\Omega z$.  Define
\begin{equation}\label{eq:branch-power}
 z^\beta_\Omega=\exp(\beta\Log_\Omega z).
\end{equation}

\begin{proposition}[Compact-uniform rational approximation]
\label{prop:rational-approximation}
Let $\beta_k\in\Q$ converge to $\beta\in\R$.  On every compact
$K\Subset\Omega$,
\begin{equation}\label{eq:uniform-power}
 \sup_{z\in K}|z^{\beta_k}_\Omega-z^\beta_\Omega|\longrightarrow0.
\end{equation}
Consequently, for fixed matrices $H_0,V$ and for $\lambda$ in a compact set,
the pencils and their determinants converge uniformly on $K$.
\end{proposition}

\begin{proof}
The image $\Log_\Omega(K)$ is compact.  The exponential function is uniformly
continuous on a compact neighborhood of
$\{\beta_k\Log_\Omega z:z\in K,\ k\gg1\}$.  Since
$(\beta_k-\beta)\Log_\Omega z\to0$ uniformly on $K$,
\eqref{eq:uniform-power} follows.  Matrix multiplication and the finite
determinant are polynomial operations in the entries, so their convergence
is uniform as well.
\end{proof}

One can now transfer isolated spectral zeros without assigning a limiting
genus.

\begin{corollary}[Stability of local spectral zeros]\label{cor:zero-stability}
Fix $\lambda\in\C$ and let
\begin{equation}\label{eq:D-beta}
 D_\beta(z)=\det\bigl[zI-H_0-\lambda z^\beta_\Omega V\bigr].
\end{equation}
Let $U\Subset\Omega$ have a piecewise smooth boundary and suppose
$D_\beta$ has no zero on $\partial U$.  Then, for all sufficiently large
$k$, $D_{\beta_k}$ and $D_\beta$ have the same number of zeros in $U$,
counted with multiplicity.  In particular, a simple zero of $D_\beta$ is
approximated by a unique zero of $D_{\beta_k}$.
\end{corollary}

\begin{proof}
The minimum of $|D_\beta|$ on $\partial U$ is positive.  Uniform convergence
from \cref{prop:rational-approximation} gives
$|D_{\beta_k}-D_\beta|<|D_\beta|$ there for large $k$.  Rouche's theorem
gives the zero count.  Applying the same argument in a small disk containing
only one simple zero gives the last statement.
\end{proof}

The denominators of rational approximants to an irrational number must be
unbounded.  Their generic algebraic genera therefore tend to infinity along
any sequence satisfying \eqref{eq:generic-locus}.  This does not define a
compact curve of infinite genus, nor does it imply convergence of global
period matrices or spectral measures.  What converges is the analytic pencil
on each fixed logarithmic branch and, under the stated boundary condition,
its local zero divisor.  The distinction prevents a local approximation
theorem from being promoted into an unsupported global limiting geometry.

\begin{figure}[H]
\centering
\includegraphics[width=\textwidth]{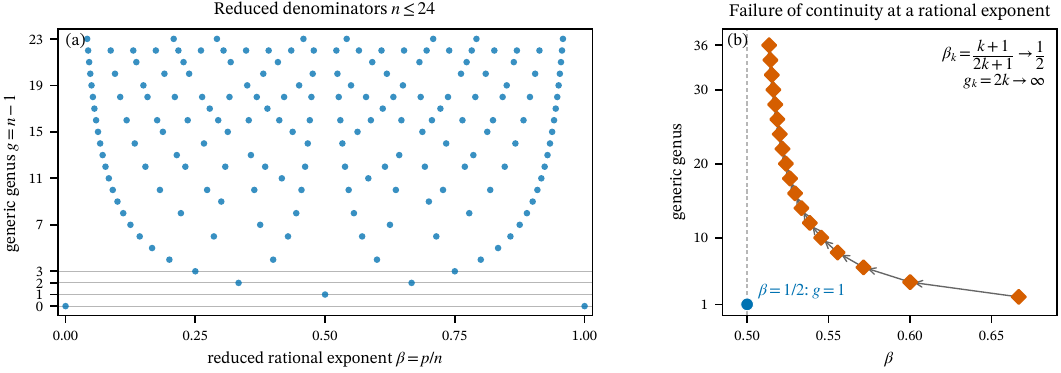}
\caption{The denominator--genus hierarchy on reduced rational exponents.
The map is highly discontinuous: while $\beta_k=(k+1)/(2k+1)$ converges to
$1/2$, its generic genus is $2k$ and diverges, whereas the reduced exponent
$1/2$ itself has genus one.  This global discontinuity is compatible with
compact-uniform convergence on a fixed logarithmic branch.  The plot displays
a consequence of the theorem; it is not numerical evidence for it.}
\label{fig:denominator-hierarchy}
\end{figure}

\section{Elliptic Moduli}\label{sec:elliptic}

The denominator $n=2$ is the first non-rational layer.  Its determinant
curve is the binary quartic
\begin{equation}\label{eq:elliptic-general-quartic}
 y^2=A u^4+B u^2+C,
\end{equation}
where $A,B,C$ are given by \eqref{eq:ABC}.  The exponent determines that
the generic curve is elliptic.  The coefficient matrices determine which
elliptic curve it is.

The half-order determinant, its moduli formula, and the canonical
specialization appeared in preliminary form in \cite{Liu2026LocalHP}.  Here
they are placed under the arbitrary-denominator theorem; the genus formula
and degeneration analysis for general reduced denominators are the new layer.

\subsection{Binary-Quartic Invariants}

For a binary quartic
\begin{equation}\label{eq:binary-quartic}
 f(U,W)=A U^4+B U^2W^2+C W^4,
\end{equation}
use the invariant convention
\begin{equation}\label{eq:IJ}
 I=B^2+12AC,\qquad
 K=72ABC-2B^3.
\end{equation}
On the nonsingular locus, the corresponding genus-one curve has
\begin{equation}\label{eq:j-binary-quartic}
 j=16\frac{I^3}{AC(B^2-4AC)^2}.
\end{equation}
This normalization follows, for example, by passing to the Jacobian model
$Y^2=X^3-27IX-27K$; a direct derivation is included in
\cref{app:binary-quartics}.  The invariant convention is standard binary-
quartic invariant theory \cite{Fisher2008}.

Substitution gives a formula in the matrix invariants alone.

\begin{theorem}[Elliptic moduli map]\label{thm:j-map}
On the locus \eqref{eq:generic-locus} with $n=2$, the determinant curve has
\begin{equation}\label{eq:j-map}
 j=
 16\,
 \frac{\bigl(\tau^2\delta^2-12\Delta_Vb^2\bigr)^3}
 {\Delta_V^2b^4
  (\tau^2-4\Delta_V)(\delta^2+4\Delta_V)}.
\end{equation}
The energy scale $E$ cancels.  The complexified parameter map to the
$j$-line is nonconstant and hence has Zariski-dense image.
\end{theorem}

\begin{proof}
From \eqref{eq:ABC} and \eqref{eq:branch-discriminant},
\begin{align}
 I&=16E^2(\tau^2\delta^2-12\Delta_Vb^2),\label{eq:I-factor}\\
 B^2-4AC&=64E^2\Delta_Vb^2.\label{eq:D-factor}
\end{align}
Inserting these identities into \eqref{eq:j-binary-quartic} cancels $E^6$
and gives \eqref{eq:j-map}.

To see nonconstancy, take the algebraic slice
$\delta=\Delta_V=1$ and let $\tau$ vary.  It is realized after
complexification by
\[
 a=\frac{\tau+1}{2},\qquad
 d=\frac{\tau-1}{2},\qquad
 b^2=\frac{5-\tau^2}{4}.
\]
On this slice,
\begin{equation}\label{eq:j-slice}
 j(\tau)=
 256\frac{(4\tau^2-15)^3}
 {5(\tau^2-5)^2(\tau^2-4)},
\end{equation}
which is nonconstant.  The image of a nonconstant rational map between
irreducible curves is Zariski dense.
\end{proof}

The theorem does not claim that every complex $j$ is attained by real
coefficients, nor that a given $j$ has a unique preimage.  It says that the
elliptic modulus varies nontrivially within a fixed half-order response
class.  Thus the exponent $-1/2$ cannot select $j=1728$ by itself.

\subsection{Canonical Specialization}

At the point $a=b=d=E=1$,
\begin{equation}\label{eq:canonical-invariants}
 (\tau,\delta,\Delta_V)=(2,0,2),
\end{equation}
and \eqref{eq:elliptic-general-quartic}, after dividing the square coordinate
by four, is the quartic already found in \cref{sec:two-channel}:
\[
 v^2=2-u^4.
\]

\begin{proposition}[Canonical elliptic curve]\label{prop:canonical-curve}
The smooth projective model of $v^2=2-u^4$ is isomorphic over $\Q$ to
\begin{equation}\label{eq:E0}
 E_0:\qquad Y^2=X^3+8X.
\end{equation}
It has $j=1728$, discriminant $-32768$, conductor $256$, Cremona label
\texttt{256b2}, and LMFDB label \texttt{256.b2}.
\end{proposition}

\begin{proof}
The quartic has the rational point $(1,1)$.  The rational map
\begin{equation}\label{eq:canonical-forward}
 X=-\frac{2(u^2+v-2)}{(u-1)^2},\qquad
 Y=\frac{4(-u^3+uv-2v+2)}{(u-1)^3}
\end{equation}
sends a dense open subset of $v^2=2-u^4$ to \eqref{eq:E0}.  Its rational
inverse is recorded and checked in \cref{app:binary-quartics}.  Hence the
smooth projective curves are isomorphic over $\Q$.  Substitution in
\eqref{eq:j-map} gives $j=1728$.  The discriminant, conductor, and labels
follow from the integral model and agree with the LMFDB entry
\cite{LMFDB256b2}.
\end{proof}

The order-four symmetry $(u,v)\mapsto(iu,-v)$, which fixes the two points at
infinity of the smooth quartic, is another quick explanation of $j=1728$ over
$\C$.  What matters for the present argument is how this symmetry entered: it
is a property of \eqref{eq:canonical-invariants}.
Generic coefficients at the same response exponent have neither this
automorphism nor this $j$-invariant.

The elliptic layer therefore illustrates both sides of the denominator
theorem.  The denominator fixes the genus, a coarse topological invariant.
The coefficient matrices retain the finer moduli.  A conformal response can
supply the first datum through a scaling dimension; it cannot, without additional
microscopic input, fix the second.

\section{Near-\texorpdfstring{$AdS_2$}{AdS2} Conformal Response}\label{sec:ads2}

We now apply the genus theorem to a narrow conformal-response problem.  The
input is an exact leading infrared monomial model.  It is not a complete
finite-temperature Green function or the global analytic continuation of a
physical response.

The power-law infrared Green function, its branch point, and its covering
space already occur in the standard near-$AdS_2$ matching analysis
\cite{FaulknerEtAl2011}.  The result below is the further, conditional step of
placing that monomial in the present two-channel determinant family.

\subsection{Conformal Scaling Limit}

Let the exact infrared model of a retarded response be
\begin{equation}\label{eq:conformal-response}
 \Sigma_{\rm IR}^R(\omega)
 =c_h[-i(\omega+i0)]^{2h_{\rm probe}-1},
 \qquad c_h\ne0,
\end{equation}
with the Fourier convention \eqref{eq:transform-conventions}.  On the
Laplace half-plane this is the boundary value of
\begin{equation}\label{eq:conformal-laplace-model}
 \Sigma_{\rm IR}(z)=c_hz^{2h_{\rm probe}-1}.
\end{equation}
Within the complexified determinant family, the nonzero coefficient and
retarded phase can be absorbed into the coupling coordinate.  They therefore
do not change the function field.  This change need not preserve a real
fixed-coupling slice or its pointwise Krein symmetry.  The response exponent
seen by the finite probe is
\begin{equation}\label{eq:beta-h}
 \beta=2h_{\rm probe}-1.
\end{equation}
For $0<h_{\rm probe}<1/2$, the monomial follows from the ordinary causal
kernel $\theta(t)t^{-2h_{\rm probe}}$.  For other dimensions, its
nonanalytic part requires distributional continuation and a specified
contact-term subtraction.

\begin{theorem}[Conformal probe genus]\label{thm:conformal-probe}
Assume the probe self-energy is exactly \eqref{eq:conformal-laplace-model},
the probe has the two-channel form \eqref{eq:intro-general-pencil}, and its
coefficients satisfy \eqref{eq:generic-locus}.  If
$2h_{\rm probe}=m/n$ is reduced, then the normalized determinant curve of
the leading infrared model, with the complexified coupling coordinate varied,
has
\begin{equation}\label{eq:conformal-genus}
 g_{\rm probe}=n-1=\den(2h_{\rm probe})-1.
\end{equation}
If $2h_{\rm probe}$ is irrational, the monomial response has infinite
monodromy and no compatible finite-sheeted response cover.
\end{theorem}

\begin{proof}
When $2h_{\rm probe}=m/n$ is reduced,
\[
 \beta=2h_{\rm probe}-1=\frac{m-n}{n}.
\]
Since $\gcd(m-n,n)=\gcd(m,n)=1$, the reduced denominator of $\beta$ is still
$n$.  Equation \eqref{eq:conformal-genus} follows from
\cref{thm:denominator-genus}.  The irrational statement follows from
\cref{thm:irrational-monodromy} with the coordinate qualification following
that theorem.
\end{proof}

In near-$AdS_2$ applications, an infrared operator dimension may determine a
leading nonanalytic power of a boundary response.  The theorem assigns a
genus only to the determinant curve obtained by replacing that leading term
with the exact monomial \eqref{eq:conformal-laplace-model}.  A mere asymptotic
relation
\begin{equation}\label{eq:conformal-asymptotic-only}
 G_R(\omega)\sim c_h[-i(\omega+i0)]^{2h_{\rm probe}-1}
\end{equation}
does not determine the global genus of the completed response.  Subleading
powers, analytic terms, additional thresholds, and ultraviolet completion
may change or eliminate a finite algebraic curve.  Nor does
\eqref{eq:conformal-genus} assign a genus to the conformal field theory or to
the spacetime geometry.  No microscopic near-$AdS_2$ construction of the
two-channel coefficient matrix, and no derivation of the genericity condition
\eqref{eq:generic-locus}, is supplied here.

\subsection{Sachdev--Ye--Kitaev Scaling}

For the fundamental Majorana fermion in the leading large-$N$,
strong-coupling conformal saddle of the even-$q>2$ Sachdev--Ye--Kitaev model,
the scaling dimension is
\begin{equation}\label{eq:syk-delta}
 \Delta_\psi=\frac1q.
\end{equation}
The large-$N$ limit is taken before the zero-temperature limit.  Maldacena
and Stanford give the conformal time-domain retarded function
\cite{MaldacenaStanford2016}.  At zero temperature it is proportional to
\begin{equation}\label{eq:syk-zero-temperature-time}
 2b\cos(\pi\Delta_\psi)\,
 \theta(t)t^{-2\Delta_\psi}.
\end{equation}
With the Fourier sign fixed in \eqref{eq:transform-conventions}, direct
integration for $\omega\ne0$, or equivalently the retarded Abel limit, gives
\begin{equation}\label{eq:syk-zero-temperature-frequency}
 2b\cos(\pi\Delta_\psi)\Gamma(1-2\Delta_\psi)
 [-i(\omega+i0)]^{2\Delta_\psi-1}.
\end{equation}
The restriction $q>2$ ensures $0<\Delta_\psi<1/2$, hence local integrability
at $t=0$.  At infinity the Fourier integral is conditionally convergent for
$\omega\ne0$, not absolutely convergent; at $\omega=0$ it diverges, and the
display records the singular retarded boundary behavior.  The frequency-space
retarded function in \cite{MaldacenaStanford2016} contains an additional
conventional factor $-i$; it changes the overall phase, not the exponent.

\begin{corollary}[SYK probe hierarchy]\label{cor:syk-genus}
Under the hypotheses of \cref{thm:conformal-probe}, for the fundamental
fermion of the interacting even-$q>2$ model,
\begin{equation}\label{eq:syk-genus}
 g_{\rm probe}=\frac{q}{2}-1.
\end{equation}
In particular, $q=4$ gives genus one for the exact leading monomial probe
model.
\end{corollary}

\begin{proof}
The reduced denominator of $2\Delta_\psi=2/q$ is $q/2$ for even $q$.  Apply
\cref{thm:conformal-probe}.
\end{proof}

The symbol $h_{\rm probe}$ denotes the dimension of the response coupled to
the probe.  In the corollary it is specialized to the fundamental-fermion
dimension $\Delta_\psi=1/q$.  It must not be confused with the dimensions
$h_m$ in the SYK bilinear and four-point spectrum.  Those dimensions are
determined by a different kernel equation and are generally not $1/q$.  If
one of them is used as a probe dimension, it has its own rational denominator
or, when irrational, the infinite monodromy described in
\cref{sec:irrational}.

\subsection{Finite-Temperature Response}

At finite inverse temperature $\beta_T$, the conformal retarded propagator in
time is proportional to
\begin{equation}\label{eq:finite-T-retarded}
 \theta(t)\left[
 \frac{\pi}{\beta_T\sinh(\pi t/\beta_T)}
 \right]^{2\Delta_\psi}.
\end{equation}
For $0<\Delta_\psi<1/2$, its transform in our convention is proportional to
\begin{equation}\label{eq:finite-T-beta-function}
 \left(\frac{2\pi}{\beta_T}\right)^{2\Delta_\psi-1}
 B\left(
 \Delta_\psi-\frac{i\beta_T\omega}{2\pi},
 1-2\Delta_\psi
 \right).
\end{equation}
This is meromorphic in frequency, with a pole tower at
\begin{equation}\label{eq:finite-T-poles}
 \omega_k=-\frac{2\pi i}{\beta_T}(\Delta_\psi+k),
 \qquad k=0,1,2,\ldots.
\end{equation}
It is not a pure power on a finite cyclic cover.  Consequently,
\eqref{eq:syk-genus} is a zero-temperature statement about the joint
determinant family of the leading monomial model, not a finite-temperature
topology theorem or a fixed-coupling spectrum.

For $q=4$, the exponent in \eqref{eq:syk-zero-temperature-frequency} is
$-1/2$, so the generic two-channel monomial model is elliptic.  Nothing in
$\Delta_\psi=1/4$ fixes the probe invariants
$(\tau,\delta,\Delta_V)$.  It therefore fixes neither the $j$-invariant nor a
particular elliptic curve.  The canonical value $j=1728$ in
\cref{prop:canonical-curve} is one matrix specialization in a nonconstant
moduli family.  The exponent likewise determines no entropy, black-hole
charge, or complete quasinormal-mode spectrum.

\section{Local Spectral Interpolation}\label{sec:interpolation}

The determinant curve answers where the pencil is singular.  Away from that
curve, its inverse answers a different question: which two-dimensional
characteristic polynomials can be obtained by evaluating the resolvent?  The
half-order family reduces this inverse problem to a quadratic equation.

The half-order trace--determinant calculation and the interpolation
quadratic appeared in preliminary form in \cite{Liu2026LocalHP}.  They are
recalled here in inverse-problem form to delimit what prescribed local data
the pencil can carry; the general arbitrary-denominator geometry is
developed in \cref{sec:rational-exponents,sec:genus}.

\subsection{Trace--Determinant Map}

Fix $E=1$ and retain a general real $J$-self-adjoint matrix
\begin{equation}\label{eq:interpolation-general-V}
 V=\begin{pmatrix}a&b\\-b&d\end{pmatrix}.
\end{equation}
On the half-order cover, put
\begin{equation}\label{eq:interpolation-general-pencil}
 \cA_V(u,\lambda)=u^2I_2-J-\frac{\lambda}{u}V,
 \qquad
 P_V(u,\lambda)=u^2\det\cA_V(u,\lambda).
\end{equation}
Direct expansion gives
\begin{equation}\label{eq:interpolation-P}
 P_V=u^6-\tau\lambda u^3-u^2-\delta\lambda u
 +\Delta_V\lambda^2.
\end{equation}
Whenever $P_V\ne0$, let $R_V=\cA_V^{-1}$.  The raw adjugate is
\begin{equation}\label{eq:interpolation-adjugate}
 \operatorname{adj}\cA_V=
 \begin{pmatrix}
 u^2+1-d\lambda/u&b\lambda/u\\
 -b\lambda/u&u^2-1-a\lambda/u
 \end{pmatrix},
\end{equation}
and therefore
\begin{equation}\label{eq:general-R-tr-det}
 \tr R_V=\frac{u(2u^3-\tau\lambda)}{P_V},
 \qquad
 \det R_V=\frac{u^2}{P_V}.
\end{equation}

\begin{theorem}[Master interpolation quadratic]\label{thm:master-interpolation}
Let $t,\kappa\in\C$ with $\kappa\tau\ne0$, and put $Y=u^2$.  On $u\ne0$, the
conditions
\begin{equation}\label{eq:general-matching}
 \tr R_V=t,
 \qquad
 \det R_V=\kappa
\end{equation}
are equivalent to
\begin{equation}\label{eq:general-lambda-solution}
 \lambda=\frac{2u^3}{\tau}-\frac{tu}{\kappa\tau}
\end{equation}
and
\begin{equation}\label{eq:master-quadratic}
 A_2Y^2+B_2Y+C_2=0,
\end{equation}
where
\begin{align}
 A_2&=\tau^2-4\Delta_V,\notag\\
 B_2&=-\frac{t(\tau^2-4\Delta_V)}\kappa+2\delta\tau,
 \label{eq:master-coefficients}\\
 C_2&=\tau^2-\frac{t\delta\tau}\kappa
 -\frac{\Delta_Vt^2}{\kappa^2}+\frac{\tau^2}{\kappa}.
 \notag
\end{align}
Every nonzero root of \eqref{eq:master-quadratic}, together with either
square root $u$ and \eqref{eq:general-lambda-solution}, gives a regular
resolvent satisfying \eqref{eq:general-matching}.  In particular,
\begin{equation}\label{eq:general-prescribed-polynomial}
 \det(I_2-TR_V)=1-tT+\kappa T^2.
\end{equation}
\end{theorem}

\begin{proof}
The determinant equation in \eqref{eq:general-matching} is
$P_V=u^2/\kappa$.  Inserting it in the trace equation
\eqref{eq:general-R-tr-det} yields
\eqref{eq:general-lambda-solution}.  Substitution into $P_V=u^2/\kappa$,
cancellation of $u^2$, and collection in $Y=u^2$ gives
\eqref{eq:master-quadratic}.  Each step is reversible for
$\kappa\tau u\ne0$.
Finally, every two-by-two matrix satisfies
\[
 \det(I_2-TR_V)=1-T\tr R_V+T^2\det R_V.
\]
\end{proof}

The theorem gives an algebraic solvability criterion, not unconditional
surjectivity for every coefficient matrix.  A nonzero root of the master
quadratic is required, and its degree can drop on special parameter loci.
The canonical real slice below is useful because one root is positive
throughout the Hasse region.

\subsection{Canonical Real Slice}

Return to the matrix $V_c$ of \eqref{eq:canonical-data}.  Write
\begin{equation}\label{eq:R-def}
 R(u,\lambda)=\cA_c(u;\lambda)^{-1}
\end{equation}
off the determinant curve.  Here $(\tau,\delta,\Delta_V)=(2,0,2)$, and
\eqref{eq:general-R-tr-det} becomes
\begin{equation}\label{eq:R-tr-det}
 \tr R=\frac{2u(u^3-\lambda)}{P_c(u,\lambda)},
 \qquad
 \det R=\frac{u^2}{P_c(u,\lambda)}.
\end{equation}

\begin{theorem}[Canonical real interpolation]\label{thm:real-interpolation}
Let $\kappa>1$ and $t\in\R$ satisfy $|t|\leq2\sqrt\kappa$.  Put
\begin{equation}\label{eq:D-tkappa}
 D_{t,\kappa}=4\kappa(\kappa+1)-t^2,
 \qquad
 Y_\pm=\frac{t\pm\sqrt{D_{t,\kappa}}}{2\kappa}.
\end{equation}
Then $Y_+>0$ and $Y_-<0$.  With
\begin{equation}\label{eq:interpolation-basepoint}
 u=\sqrt{Y_+},
 \qquad
 \lambda=u^3-\frac{tu}{2\kappa},
\end{equation}
the resolvent is regular and satisfies
\begin{equation}\label{eq:interpolation-result}
 \tr R=t,
 \qquad
 \det R=\kappa,
 \qquad
 \det(I_2-TR)=1-tT+\kappa T^2.
\end{equation}
\end{theorem}

\begin{proof}
In the canonical specialization, \eqref{eq:master-quadratic} is equivalent
to
\begin{equation}\label{eq:canonical-master-quadratic}
 2\kappa^2Y^2-2\kappa tY+t^2-2\kappa(\kappa+1)=0,
\end{equation}
whose roots are \eqref{eq:D-tkappa}.  The assumed bound gives
$\sqrt{D_{t,\kappa}}\geq2\kappa>|t|$, hence the stated signs.  The positive
root is nonzero and $P_c=Y_+/\kappa\ne0$.  The remaining identities follow from
\cref{thm:master-interpolation}.
\end{proof}

If $E/\Q$ is an elliptic curve with good reduction at a prime $p$, the Hasse
bound $|a_p|\leq2\sqrt p$ permits the substitution
\begin{equation}\label{eq:euler-input}
 (t,\kappa)=(a_p,p).
\end{equation}
It gives
\begin{equation}\label{eq:euler-interpolation}
 \det(I_2-p^{-s}R_p)
 =1-a_pp^{-s}+p^{1-2s}.
\end{equation}
Both coefficients in the right-hand side were supplied in
\eqref{eq:euler-input}; the identity is inverse interpolation.

\begin{proposition}[Input non-selection]\label{prop:input-nonselection}
For every assignment of real numbers $t_p$ satisfying
$|t_p|\leq2\sqrt p$, the same canonical pencil has prime-dependent real
basepoints for which
\begin{equation}\label{eq:arbitrary-hasse-interpolation}
 \det(I_2-TR_p)=1-t_pT+pT^2.
\end{equation}
At each such basepoint, $R_p$ is conjugate over $\C$ to
\begin{equation}\label{eq:companion-matrix}
 C_p=\begin{pmatrix}0&-p\\1&t_p\end{pmatrix}.
\end{equation}
Consequently, the fixed pencil and the existence of matching basepoints do
not select an arithmetic trace sequence.
\end{proposition}

The companion realization of a prescribed polynomial is classical; explicit
structured constructions provide one representative reference
\cite{ChanCorless2016}.  Here the additional calculation only locates such a
realization inside the fixed pencil family.

\begin{proof}
Apply \cref{thm:real-interpolation} independently at each $p$.  The matrices
$R_p$ and $C_p$ have the same characteristic polynomial.  If
$t_p^2\ne4p$, both are diagonalizable with the same two eigenvalues.  If
$t_p^2=4p$, the selected basepoint has
\[
 \frac{\lambda_p}{u_p}
 =Y_+-\frac{t_p}{2p}=1,
\]
so the off-diagonal entries of $\cA_c$ and $R_p$ are nonzero.  Neither matrix
is scalar; both have the same single Jordan block and are again conjugate.
\end{proof}

\subsection{Inverse and Forward Problems}

The information flow in this section is one-directional.  The determinant
curve constrains singular evaluations and organizes analytic continuation.
It does not determine the target polynomial supplied to an inverse problem:
in \eqref{eq:euler-interpolation} the prime enters only through the supplied
coefficients, and no prime sequence, Euler product, or global operator
follows from the identity.  A forward arithmetic construction would have to
produce a trace sequence from fixed operator data or from specified
dynamics.  Solving for a basepoint after the coefficients are known instead
characterizes local response matrices compatible with prescribed
characteristic data.

\section{Discussion}\label{sec:discussion}

\subsection{Topology and Channel Dimension}

The calculation begins with a finite matrix but does not end with the usual
finite-dimensional characteristic polynomial.  The reason is now explicit.
The response $z^{r/n}$ first requires the cyclic cover $z=u^n$.  The
two-channel determinant is then quadratic in the combined coordinate
$q=\lambda u^r$.  Its discriminant supplies a second cover, branched at
$2n$ points on the generic locus.  The genus $n-1$ is the Riemann--Hurwitz
record of these two operations.

Neither part can be discarded.  The reduced denominator fixes the order of
the response monodromy, but the two-channel determinant supplies the double
cover.  A one-channel pencil has only a rational graph for $\lambda$ and does
not obey the same genus formula.  For three or more channels the determinant
is generically of higher degree in $q$, and its branch cycles need not reduce
to a hyperelliptic problem.  Its unmarked generic genus belongs to classical
matrix-polynomial spectral-curve theory, as recalled below.  The present
detailed result concerns the two-channel class, including the preferred
coupling coordinate and its boundary data, rather than a new
dimension-independent genus law.

The curve uses $\lambda$ as a variable complexified coupling coordinate.  If
one fixes a physical coupling $\lambda_0$, the remaining zeros in $z$ form a
fiber of this family and do not carry the genus $n-1$.  The theorem classifies
the joint determinant geometry of the pencil, not the discrete spectrum of a
single fixed-coupling operator.

The numerator has a subtler role.  It disappears from the normalized
function field after the birational coordinate change, but it remains in the
divisor of the original coupling parameter and in the boundary singularity
of a raw plane closure.  Equations \eqref{eq:lambda-divisor-orders},
\eqref{eq:negative-r-closure}, and \eqref{eq:positive-r-closure} make this
distinction concrete.  Topological equivalence of normalizations does not
mean equality of their preferred embeddings or observables.

\subsection{Higher-Channel Context}\label{sec:higher-channel-context}

The generic genus has a short higher-channel interpretation.  Let
$J=\diag(j_1,\ldots,j_N)$ with $j_i\in\{1,-1\}$, let
$H_0=\diag(E_1,\ldots,E_N)$ have distinct nonzero entries, and take
$V=JS$ with $S^\top=S$.  On the rational response cover, put
$w=\lambda u^r$.  The active determinant curve is
\begin{equation}\label{eq:higher-channel-curve}
 \det(u^nI_N-H_0-wV)=0.
\end{equation}
When $V$ is invertible, this is equivalently
\[
 \det\bigl(V^{-1}(u^nI_N-H_0)-wI_N\bigr)=0.
\]
It is therefore a size-$N$, degree-$n$ matrix-polynomial spectral curve.
On the generic regular locus, Vivolo's formula \cite{Vivolo2000} gives
\begin{equation}\label{eq:higher-channel-generic-genus}
 g_N=\frac{(N-1)(Nn-2)}{2}.
\end{equation}
Thus \eqref{eq:main-genus} is the specialization $N=2$.  Equation
\eqref{eq:higher-channel-generic-genus} is cited here as a classical
result, not claimed as a new genus theorem.  The change
$w=\lambda u^r$ also shows why the unmarked function field is independent
of the reduced numerator, while the distinguished function
$\lambda=wu^{-r}$ retains its cyclic character and divisor.

The indefinite two-energy slice gives a useful singular specialization of
the same background.

\begin{proposition}[Indefinite two-energy slice]\label{prop:higher-channel-EJ}
Let
\[
 J=\diag(I_{p_+},-I_{p_-}),\qquad
 p_+,p_-\geq1,\qquad p_++p_-=N,
\]
and set $H_0=EJ$ with $E\ne0$.  For every $n\geq2$, there is a
nonempty Zariski-open subset of matrices $V=JS$, $S^\top=S$, on which
the normalization of \eqref{eq:higher-channel-curve} is geometrically
integral and has genus
\begin{equation}\label{eq:higher-channel-EJ-genus}
 \boxed{g_{EJ}=np_+p_--N+1.}
\end{equation}
\end{proposition}

\begin{proof}
First consider the degree-$N$ projective base curve
\[
 B:\quad \det(XI_N-EZJ-WV)=0.
\]
On $W=0$ it has the two forced points
$P_+=(E:0:1)$ and $P_-=(-E:0:1)$.  Relative to the positive and negative
blocks of $J$, their tangent cones are, up to nonzero factors,
\[
 \det(\xi I_{p_+}-WV_{++}),\qquad
 \det(\eta I_{p_-}-WV_{--}).
\]
For generic $V$, the two diagonal blocks are regular semisimple.  Hence
$P_+$ and $P_-$ are ordinary multiple points of multiplicities $p_+$ and
$p_-$.  The required generic locus is nonempty: start with a diagonal
$V$, whose same-sign eigenvalue lines meet only at the two forced points,
and perturb its symmetric representative $S=JV$ by generic off-diagonal
entries.  This smooths the cross-sign intersections while preserving the
two ordinary multiple points.  The cross-sign incidence graph is connected,
so smoothing all its edges makes the nearby fiber irreducible.  No further
singularity persists on a generic projective fiber.

The normalization $B^\nu$ consequently has genus
\begin{align*}
 g(B^\nu)
 &=\frac{(N-1)(N-2)}2
   -\frac{p_+(p_+-1)+p_-(p_--1)}2 \\
 &=(p_+-1)(p_--1).
\end{align*}
Equivalently, this is Vivolo's normal-crossing correction after the
projective change $t=1/W$, $y=X/W$, which writes the base curve as
$\det(V+tEJ-yI_N)=0$.

On the same generic locus, $x=X/Z$ has $N$ simple zeros and $N$ simple
poles on $B^\nu$.  Equation \eqref{eq:higher-channel-curve} is the
cyclic pullback $u^n=x$, totally ramified at those $2N$ points and nowhere
else.  Riemann--Hurwitz gives
\[
 2g_{EJ}-2
 =n\bigl(2(p_+-1)(p_--1)-2\bigr)+2N(n-1),
\]
which reduces to \eqref{eq:higher-channel-EJ-genus}.
\end{proof}

The proposition is a structured consequence of the classical singular
spectral-curve genus correction and a cyclic-cover calculation.  It is not a
classification of repeated channel energies, singular $V$, or reducible
strata.  Definite signature is excluded: then $H_0$ is scalar and the base
determinant generically splits into eigenvalue lines after diagonalizing
$V$.

\subsection{Finite Channels and Nonlocal Memory}

The channel space in \eqref{eq:volterra} is finite-dimensional.  Its state at
one time nevertheless does not determine its future without a memory
history.  The power law is the Laplace transform of a causal kernel with a
long tail, or the limit of eliminating increasingly broad continuum baths.
The finite matrix is therefore not a finite-state Markovian system in
disguise.  Its spectral curve retains analytic information about degrees of
freedom that have been eliminated, while remaining a joint determinant family
rather than a fixed-coupling spectrum.

The cutoff construction in \cref{sec:causal-response} also marks the operator
boundary of the model.  At every finite cutoff, the Schur complement is an
ordinary bounded algebraic elimination; with an indefinite bath signature it
also has the reciprocal Krein realization \eqref{eq:krein-bath-generator}.
The exact uncut measure has
infinite total mass, so the constant coupling vector leaves the bath Hilbert
space.  The limiting power law is proved here only as a locally uniform
Stieltjes limit of cutoff Schur complements.  Constructing it as an uncut
singular or form-coupled parent operator would require additional domain and
closedness results.  This distinction matters whenever spectral properties of
a parent operator are to be inferred from the reduced pencil.

\subsection{Local and Global Analytic Structure}

The rational and irrational cases meet in a deliberately nonuniform way.
On any fixed logarithmic domain away from zero and at fixed coupling, rational exponents converge
to an irrational exponent together with all local determinant data needed
for Rouche's theorem.  Isolated spectral zeros are stable there.  At the same
time, the reduced denominators of rational approximants are unbounded, so the
genera of their generic compact curves tend to infinity.  Local analytic
convergence therefore coexists with a discontinuous global topological
classification.

This is not a paradox.  Increasing denominators add sheets that record more
global monodromy circuits, whereas a fixed logarithmic chart does not identify
the resulting compact global covers.  The determinant branch points need not
approach the monodromy center.  It would be incorrect to infer a limiting
compact surface or a convergent period matrix from
\cref{prop:rational-approximation}.  The stable object is the local analytic
pencil and its zero divisor on the chosen domain.

The degenerate strata give a second form of nonuniformity.  The denominator
sets a generic upper value, while matrix coefficients can collide branch
points, lower the degree, or split the curve.  The odd-multiplicity
square-class part determines the normalized genus; the full multiplicity
data determine whether the plane model has a node, cusp, tacnode, or
nonreduced component.  Both are needed for a spectral interpretation.

\subsection{Conformal Probe Models}

In an exact monomial infrared model, a conformal dimension fixes the reduced
response denominator and hence the generic genus of a two-channel probe.
This converts an analytic scaling exponent into a topological invariant of a
specified reduced model.  It does not convert an asymptotic expansion into a
global algebraic curve.  A physical response may contain contact terms,
several fractional thresholds, logarithms, or a meromorphic thermal pole
tower.  Each can change the completion on which genus is defined.

The SYK specialization illustrates the boundary sharply.  The fundamental
zero-temperature conformal fermion has $\Delta_\psi=1/q$, so the exact leading
monomial model gives $g_{\rm probe}=q/2-1$ for interacting even $q$.  The
bilinear tower has different dimensions, and the finite-temperature response
is a beta-function ratio.  Neither is described by simply retaining the same
genus.  The result is therefore a controlled probe statement adjacent to
near-$AdS_2$ physics, not a theorem about the full gravitational spectrum.

\subsection{Moduli and Interpolation}

The denominator controls a coarse invariant.  It does not fix the point in
moduli space.  At $n=2$, \cref{thm:j-map} gives a nonconstant rational
$j$-map in the matrix coefficients.  The value $1728$ and the curve
$Y^2=X^3+8X$ occur on one explicit slice.  Other coefficients with the same
half-order response give other elliptic moduli or lie on degenerate strata.

The interpolation results are complementary.  They show that resolvent
evaluations can carry arbitrary prescribed quadratic data on a broad real
domain.  Their very flexibility prevents them from being a
selection principle.  The coefficient sequence enters through the chosen
basepoints.  A forward construction would need additional dynamics or a
global compatibility condition that produces those basepoints rather than
solving for them after the target is supplied.

\subsection{Extensions}

Several extensions are directly testable without changing the conceptual
separation used here.  A finite sum of rational powers lives on a cyclic
cover determined by the least common denominator, but its determinant branch
polynomial is more general than \eqref{eq:Qn-repeat}; its genus must be
recomputed from the actual discriminant.  Higher channel number replaces the
quadratic extension by a higher-degree covering and requires its full branch
cycle data.  Rational or polynomial corrections on the same response cover
remain algebraic, but can add branch points and change genus.  Irrational
powers, logarithms, and thermal gamma functions generally leave the class of
finite algebraic response covers.

These tests clarify what is robust.  Causal elimination explains the origin
of multivalued coefficients.  Algebraic normalization determines the global
topology only after the exact response and coefficient locus are fixed.
Neither step by itself supplies arithmetic selection or a complete
black-hole spectral theory.

\section{Outlook}\label{sec:outlook}

A separate problem concerns quasinormal-mode equations whose exact or
semiclassical analysis is organized by Seiberg--Witten curves.  The natural
operator question is whether such a curve can be derived from a fixed
Feshbach--Schur reduction, with the horizon and asymptotic sectors providing
the eliminated continua and a finite set of retained channels providing the
probe block.

For this to be more than a formal similarity, four ingredients would be
needed.  One must specify a closed parent operator with the physical boundary
conditions, choose a projection independent of the spectral root, derive the
causal self-energy on a controlled domain, and prove that the reduced
determinant has the same normalized function field as the proposed
Seiberg--Witten curve.  The map must also preserve the quasinormal-mode
condition rather than only the shape of a polynomial after a
spectral-parameter-dependent substitution.

The continuum-bath construction in \cref{sec:causal-response} supplies a
model for one step of such a derivation, but it does not identify the required
black-hole parent operator or projection.  Establishing that identification
is an independent question.  No quasinormal-mode--Seiberg--Witten derivation
is claimed in the present paper.

\appendix
\section{Determinant Identities}\label{app:determinants}

This appendix records the algebraic eliminations used in the main text.  The
calculations begin with the raw matrix entries; no genus formula is assumed.

\subsection{Two-Channel Determinant}

On the reduced cyclic cover $z=u^n$, $z^{r/n}=u^r$, the pencil is
\begin{equation}\label{eq:app-lifted-matrix}
 \cA_{r/n}(u^n;\lambda)=
 \begin{pmatrix}
 u^n-E-a\lambda u^r&-b\lambda u^r\\
 b\lambda u^r&u^n+E-d\lambda u^r
 \end{pmatrix}.
\end{equation}
Multiplying the diagonal entries and adding the off-diagonal product gives
\begin{align}
 \det\cA_{r/n}
 &=(u^n-E)(u^n+E)
 -\lambda u^r\{d(u^n-E)+a(u^n+E)\}
 +(ad+b^2)\lambda^2u^{2r}\notag\\
 &=u^{2n}-E^2
 -\lambda u^r(\tau u^n+E\delta)
 +\Delta_V\lambda^2u^{2r}.
 \label{eq:app-raw-determinant}
\end{align}
The sign $+b^2$ in $\Delta_V=ad+b^2$ comes from the $J$-self-adjoint
off-diagonal pair $b,-b$.

Set $q=\lambda u^r$ and
\begin{equation}\label{eq:app-Bn}
 B_n(u)=\tau u^n+E\delta.
\end{equation}
Then the determinant locus is
\begin{equation}\label{eq:app-q-equation}
 F_n(u,q):=\Delta_Vq^2-B_n(u)q+u^{2n}-E^2=0.
\end{equation}
If $\Delta_V\ne0$, direct expansion gives the polynomial identity
\begin{equation}\label{eq:app-square-identity}
 [2\Delta_Vq-B_n(u)]^2-Q_n(u)=4\Delta_VF_n(u,q),
\end{equation}
where
\begin{equation}\label{eq:app-Qn}
 Q_n(u)=B_n(u)^2-4\Delta_V(u^{2n}-E^2).
\end{equation}
Thus $y=2\Delta_Vq-B_n$ gives the double-cover model without taking a
square root of a rational expression.  The inverse is
\begin{equation}\label{eq:app-q-inverse}
 q=\frac{y+B_n(u)}{2\Delta_V},
 \qquad
 \lambda=u^{-r}\frac{y+B_n(u)}{2\Delta_V}.
\end{equation}
These formulas hold on $u\ne0$ and extend as rational maps between the
normalizations.

When $\Delta_V=0$, \eqref{eq:app-q-equation} is linear:
\begin{equation}\label{eq:app-linear-q}
 q=\frac{u^{2n}-E^2}{\tau u^n+E\delta},
\end{equation}
provided the denominator is not identically zero.  Applying
\eqref{eq:app-square-identity} on this locus would replace a rational graph by
a noninvertible doubled equation; this is why the case is separated before
completion of the square.  For the real coefficient class used here, an
identically zero denominator on $\Delta_V=0$ means $V=0$; then the determinant
is $u^{2n}-E^2$ and the components are vertical in the coupling coordinate.

\subsection{Low-Denominator Models}

The specialization $E=1$, $(a,b,d)=(2,1,3)$ has
\begin{equation}\label{eq:app-generic-invariants}
 (\tau,\delta,\Delta_V)=(5,-1,7),
 \qquad
 (A,B,C)=(-3,-10,29).
\end{equation}
For the negative representatives $r=-1$ and $n=1,\ldots,4$, the raw
determinants are
\begin{equation}\label{eq:app-low-raw}
 F_{-1,n}=u^{2n}-1
 -\lambda u^{-1}(5u^n-1)+7\lambda^2u^{-2},
\end{equation}
while every completed model is
\begin{equation}\label{eq:app-low-Q}
 y^2=-3u^{2n}-10u^n+29.
\end{equation}
The two roots in $t=u^n$ are
\begin{equation}\label{eq:app-low-t-roots}
 t_\pm=\frac{-5\pm4\sqrt7}{3}.
\end{equation}
They are distinct and nonzero, so their $n$th roots give $2n$ simple branch
points.  The resulting genera are $0,1,2,3$.  The computation uses the same
matrix coefficients at every denominator; no parameter is retuned to force
the pattern.

\subsection{One-Channel Comparison}

For comparison, a scalar pencil
\begin{equation}\label{eq:app-one-channel}
 A^{(1)}(z;\lambda)=z-E-\lambda v z^{r/n},
 \qquad v\ne0,
\end{equation}
has, on $z=u^n$, the active determinant graph
\begin{equation}\label{eq:app-one-channel-graph}
 \lambda=\frac{u^n-E}{vu^r}.
\end{equation}
Its normalization is rational for every $n$.  The denominator--genus theorem
therefore uses both the cyclic response cover and the quadratic determinant
of the two-channel class.

\subsection{Resolvent Elimination}

For the half-order pencil with $E=1$, multiplication of
\eqref{eq:interpolation-adjugate} by $u^2/P_V$ gives
\eqref{eq:general-R-tr-det}.  To eliminate a prescribed pair $(t,\kappa)$,
first
write
\begin{equation}\label{eq:app-det-match}
 P_V=\frac{u^2}{\kappa}.
\end{equation}
The trace equation then reads
\begin{equation}\label{eq:app-trace-match}
 \frac{\kappa(2u^3-\tau\lambda)}u=t,
\end{equation}
which is \eqref{eq:general-lambda-solution}.  Substitution into
\eqref{eq:interpolation-P} yields
\begin{align}
 0={}&(\tau^2-4\Delta_V)Y^2
 +\left[-\frac{t(\tau^2-4\Delta_V)}\kappa+2\delta\tau\right]Y
 \notag\\
 &+\tau^2-\frac{t\delta\tau}\kappa
 -\frac{\Delta_Vt^2}{\kappa^2}+\frac{\tau^2}{\kappa},
 \label{eq:app-master-expanded}
\end{align}
after multiplication by a nonzero scalar.  This gives
\eqref{eq:master-quadratic} from the raw determinant and also identifies its
exceptional hypotheses: $u$, $\kappa$, and $\tau$ must be nonzero before the
elimination is reversed.

In the canonical case $(\tau,\delta,\Delta_V)=(2,0,2)$,
\eqref{eq:app-master-expanded}, multiplied by $-\kappa^2/2$, becomes
\begin{equation}\label{eq:app-canonical-master}
 2\kappa^2Y^2-2\kappa tY+t^2-2\kappa(\kappa+1)=0.
\end{equation}
Its discriminant in $Y$ is
$4\kappa^2[4\kappa(\kappa+1)-t^2]$, giving the roots in
\eqref{eq:D-tkappa}.

\section{Branch Points and Normalization}\label{app:normalization}

\subsection{Generic Compactification}

Assume $A C(B^2-4AC)\ne0$.  Let
\begin{equation}\label{eq:app-tpm}
 t_\pm=\frac{-B\pm\sqrt{B^2-4AC}}{2A}.
\end{equation}
Then
\begin{equation}\label{eq:app-Q-factor}
 Q_n(u)=A(u^n-t_+)(u^n-t_-).
\end{equation}
After choosing $n$th roots, the finite branch divisor of the projection to
$u$ is
\begin{equation}\label{eq:app-branch-divisor}
 \left\{\zeta_n^k t_+^{1/n}:0\leq k<n\right\}
 \mathbin{\cup}
 \left\{\zeta_n^k t_-^{1/n}:0\leq k<n\right\},
\end{equation}
where $\zeta_n=e^{2\pi i/n}$.  The two orbits are disjoint because
$t_+\ne t_-$, and none contains zero.

The affine equation $y^2=Q_n(u)$ is smooth: a singular point would satisfy
$y=Q_n=Q_n'=0$, contradicting square-freeness.  Its ordinary projective
plane closure is generally singular at infinity for $n>1$, so the
smooth-plane degree formula is not applicable.  Instead use the two affine
charts
\begin{align}
 y^2&=Au^{2n}+Bu^n+C,\label{eq:app-finite-chart}\\
 v^2&=A+Bt^n+Ct^{2n},
 \qquad t=u^{-1},\quad v=yu^{-n}.
 \label{eq:app-infinity-chart}
\end{align}
The second chart has two smooth points $(t,v)=(0,\pm\sqrt A)$.  Gluing these
charts and normalizing gives the smooth projective double cover used in
\cref{thm:denominator-genus}.

At a simple root $u_0$ of $Q_n$, a local coordinate $s=y$ satisfies
$u-u_0=c s^2+O(s^4)$ with $c\ne0$, so the ramification index of $\pi_u$ is
two.  At infinity, $t$ itself is a local coordinate and the two points are
unramified.  This proves directly that the branch divisor has degree $2n$.

For the map $z=u^n$, the local parameter at either point above $u=0$ gives
$z=u^n$, hence index $n$.  At infinity, $z=t^{-n}$ gives the same index.  At
each point in \eqref{eq:app-branch-divisor}, $z-z_0$ is a nonzero scalar
multiple of $s^2$ to leading order, hence index two.  These local models give
the ramification count \eqref{eq:pi-z-total-ramification} without relying on
a picture of the sheets.

\subsection{Pole Clearing and Saturation}

The operator pencil is defined on the active locus $u\ne0$ when the lifted
response has a pole.  A polynomial equation obtained by multiplying its
determinant by a power of $u$ must therefore be interpreted as a closure of
that punctured locus.  Algebraically, if $P(u,\lambda)$ is a cleared equation,
the relevant ideal is
\begin{equation}\label{eq:app-saturation}
 (P):u^\infty
 =\{f:\text{some }u^kf\text{ lies in }(P)\}.
\end{equation}
This removes any component supported entirely at the pole divisor.  Isolated
boundary points that are limits of the active locus remain and are then
resolved by normalization.

For $r=-s<0$, the cleared equation is
\begin{equation}\label{eq:app-negative-closure-repeat}
 P_-(u,\lambda)=u^{2s}(u^{2n}-E^2)
 -\lambda u^s(\tau u^n+E\delta)
 +\Delta_V\lambda^2.
\end{equation}
Its discriminant in $\lambda$ is
\begin{equation}\label{eq:app-negative-discriminant}
 \operatorname{Disc}_\lambda(P_-)=u^{2s}Q_n(u).
\end{equation}
When $C\ne0$, the two local roots have the form
\begin{equation}\label{eq:app-negative-branches}
 \lambda_\pm(u)=c_\pm u^s+O(u^{s+n}),
 \qquad c_+\ne c_-.
\end{equation}
After an analytic change of coordinates, their union is
$Y^2=X^{2s}$, an $A_{2s-1}$ singularity.  The normalization separates the
two branches.  For $s=1$ their tangents are distinct and the point is a node;
for $s>1$ the branches are tangent.

For $r=s>0$, the finite-$\lambda$ chart does not display this boundary.  In
$\mu=1/\lambda$, multiplication by $\mu^2$ gives
\begin{equation}\label{eq:app-positive-closure-repeat}
 P_+(u,\mu)=(u^{2n}-E^2)\mu^2
 -u^s(\tau u^n+E\delta)\mu
 +\Delta_Vu^{2s}.
\end{equation}
Its discriminant is again $u^{2s}Q_n(u)$, so the same singularity occurs over
$\lambda=\infty$.  This calculation is the plane-model counterpart of the
divisor orders \eqref{eq:lambda-divisor-orders}.

\subsection{Square-Class Normalization}

Suppose $Q=H^2S$ as in \eqref{eq:squareclass-factor}.  The rational map
\begin{equation}\label{eq:app-squareclass-map}
 (u,y)\longmapsto (u,Y=y/H(u))
\end{equation}
identifies the function fields of $y^2=Q$ and $Y^2=S$ whenever $S$ is
nonconstant.  Zeros of $H$ are precisely where this map resolves repeated
branch points.  The normalization is determined by $S$, but the local plane
singularity at a root of multiplicity $m$ remains analytically equivalent,
up to a unit, to
\begin{equation}\label{eq:app-local-multiplicity}
 y^2=x^m.
\end{equation}
Thus $m=2$ is a node, $m=3$ is a cusp, and higher $m$ gives the corresponding
$A_{m-1}$ singularity.  This is why replacing $Q$ by its polynomial radical
loses information in two ways: it retains even factors that should disappear
from the function field, and it erases their multiplicities as singularity
data.

\subsection{Principal Degenerations}

When $A=0$ and $BC\ne0$, the equation is
\begin{equation}\label{eq:app-A-zero}
 y^2=Bu^n+C.
\end{equation}
Its finite roots are simple.  If $n$ is odd, the single point at infinity is
ramified; if $n$ is even, there are two unramified points.  In both cases the
genus is $\lfloor(n-1)/2\rfloor$.

When $C=0$ and $AB\ne0$, the local equation at zero is
\begin{equation}\label{eq:app-C-zero-local}
 y^2=u^n(B+Au^n).
\end{equation}
Since the parenthesis is a unit, this has type $A_{n-1}$ for $n\geq2$.
Dividing out the largest square power of $u$ gives square-class degree $n$
for even $n$ and $n+1$ for odd $n$.

If $B^2-4AC=0$ with $AC\ne0$, then
\begin{equation}\label{eq:app-perfect-square}
 Q_n(u)=A(u^n-t_0)^2,
 \qquad t_0=-\frac{B}{2A}\ne0.
\end{equation}
The reduced curve consists of two rational components
$y=\pm\sqrt A(u^n-t_0)$.  They meet transversely at the $n$ roots of
$u^n=t_0$, giving $n$ ordinary nodes in the affine plane model.  In the
matrix parameterization with $E\Delta_V\ne0$, this discriminant condition is
equivalent to $b=0$.

Finally, if $A=C=0$ and $B\ne0$, then $Q_n=Bu^n$.  For odd $n$, removal of the
largest square leaves $S=u$ and one rational component.  For even $n$, the
polynomial is a square and the reduced curve has two rational components.
If $A=B=C=0$, the equation is nonreduced and does not define a smooth spectral
curve before taking its reduced components.

On $\Delta_V=0$ with $V\ne0$, the linear equation
\eqref{eq:app-linear-q} gives the main rational component.  At a common zero of
$u^{2n}-E^2$ and $\tau u^n+E\delta$, the unreduced equation can contain a
vertical component.  Such components must be stated separately rather than
passed through the noninvertible square-completion map.  When $V=0$, the
vertical components over the roots of $u^{2n}=E^2$ are the whole determinant
locus.

\section{Binary Quartics and \texorpdfstring{$j$}{j}-Invariants}
\label{app:binary-quartics}

\subsection{Invariant Convention}

For the even binary quartic
\begin{equation}\label{eq:app-binary-quartic}
 f(U,W)=AU^4+BU^2W^2+CW^4,
\end{equation}
we use
\begin{equation}\label{eq:app-IK}
 I=B^2+12AC,
 \qquad
 K=72ABC-2B^3.
\end{equation}
These are the usual binary-quartic invariants specialized to vanishing cubic
and linear coefficients.  Direct expansion gives
\begin{equation}\label{eq:app-invariant-identity}
 4I^3-K^2=432AC(B^2-4AC)^2.
\end{equation}
The quartic is nonsingular precisely when
$AC(B^2-4AC)\ne0$; its univariate discriminant is
\begin{equation}\label{eq:app-quartic-discriminant}
 \operatorname{Disc}_u(Au^4+Bu^2+C)
 =16AC(B^2-4AC)^2.
\end{equation}

One Jacobian convention for the genus-one curve $y^2=f(u,1)$ is
\begin{equation}\label{eq:app-jacobian-model}
 \mathcal E_f:\qquad
 Y^2=X^3-27IX-27K.
\end{equation}
For a Weierstrass equation $Y^2=X^3+aX+b$,
\begin{equation}\label{eq:app-weierstrass-j}
 j=1728\frac{4a^3}{4a^3+27b^2}.
\end{equation}
Applying this to \eqref{eq:app-jacobian-model} and using
\eqref{eq:app-invariant-identity} gives
\begin{equation}\label{eq:app-quartic-j}
 j=\frac{6912I^3}{4I^3-K^2}
 =16\frac{I^3}{AC(B^2-4AC)^2}.
\end{equation}
This fixes all normalization factors in \eqref{eq:j-binary-quartic}; no
comparison by numerical $j$-values is needed.

\subsection{Matrix-Parameter Substitution}

For the determinant quartic, the coefficients are
\begin{equation}\label{eq:app-quartic-ABC}
 A=\tau^2-4\Delta_V,
 \quad B=2E\tau\delta,
 \quad C=E^2(\delta^2+4\Delta_V).
\end{equation}
Using $\tau^2-\delta^2=4ad$ and $\Delta_V=ad+b^2$, one obtains
\begin{align}
 B^2-4AC&=64E^2\Delta_Vb^2,
 \label{eq:app-quartic-D-factor}\\
 I&=16E^2(\tau^2\delta^2-12\Delta_Vb^2).
 \label{eq:app-quartic-I-factor}
\end{align}
Substitution into \eqref{eq:app-quartic-j} yields
\begin{equation}\label{eq:app-full-j-map}
 j=16
 \frac{(\tau^2\delta^2-12\Delta_Vb^2)^3}
 {\Delta_V^2b^4
  (\tau^2-4\Delta_V)(\delta^2+4\Delta_V)}.
\end{equation}
The powers of $E$ cancel.  Thus changing the overall spectral scale leaves
the elliptic modulus unchanged, while changing the dimensionless matrix
invariants generally moves the curve in moduli.

The map is not constant.  On the complexified slice
$\delta=\Delta_V=1$,
\begin{equation}\label{eq:app-j-slice-realization}
 a=\frac{\tau+1}{2},
 \qquad
 d=\frac{\tau-1}{2},
 \qquad
 b^2=\frac{5-\tau^2}{4},
\end{equation}
formula \eqref{eq:app-full-j-map} becomes
\begin{equation}\label{eq:app-j-slice}
 j(\tau)=256\frac{(4\tau^2-15)^3}
 {5(\tau^2-5)^2(\tau^2-4)}.
\end{equation}
This nonconstant rational function also shows why the half-order exponent
does not select a single elliptic curve.

\subsection{Canonical Birational Maps}

The canonical quartic is
\begin{equation}\label{eq:app-canonical-quartic}
 v^2=2-u^4.
\end{equation}
On the dense open set $u\ne1$, define
\begin{equation}\label{eq:app-forward-map}
 X=-\frac{2(u^2+v-2)}{(u-1)^2},
 \qquad
 Y=\frac{4(-u^3+uv-2v+2)}{(u-1)^3}.
\end{equation}
A direct common-denominator calculation gives
\begin{align}
 Y^2-X^3-8X
 ={}&\frac{8(5u^2-8u+v+2)}{(u-1)^6}
 (u^4+v^2-2).
 \label{eq:app-forward-verification}
\end{align}
Hence \eqref{eq:app-forward-map} maps the quartic to
\begin{equation}\label{eq:app-E0-repeat}
 E_0:\qquad Y^2=X^3+8X.
\end{equation}

Conversely, on $2X+Y+8\ne0$, put
\begin{equation}\label{eq:app-inverse-map}
 u=\frac{4X+Y-8}{2X+Y+8},
 \qquad
 v=\frac{-2X^3+24X^2+Y^2+48Y+64}{(2X+Y+8)^2}.
\end{equation}
Substitution gives $v^2+u^4-2=0$ on
$Y^2=X^3+8X$.  Substituting
\eqref{eq:app-forward-map} into \eqref{eq:app-inverse-map}, and conversely,
gives the identity wherever both maps are defined.  The maps therefore
identify the smooth projective curves over $\Q$.

For \eqref{eq:app-E0-repeat}, the standard Weierstrass formulas
\cite{Silverman2009} give
\begin{equation}\label{eq:app-E0-invariants}
 \Delta=-16\cdot4\cdot8^3=-32768,
 \qquad j=1728.
\end{equation}
Its conductor $256$, Cremona label \texttt{256b2}, and LMFDB label
\texttt{256.b2} are external database data, not consequences of the response
exponent alone; they agree with the cited LMFDB record \cite{LMFDB256b2}.

\section{Computational Verification}\label{app:verification}

The accompanying verification starts from raw matrix definitions rather than
from the closed branch polynomial.  Its purpose is to catch sign,
normalization, and degeneration errors; the proofs in the paper do not depend
on floating-point output.

\subsection{Exact Scripts}

Three computer-algebra implementations are included:
\begin{equation}\label{eq:app-verification-files}
 \begin{split}
 &\texttt{verification/verify\_full.sage},\\
 &\texttt{verification/verify\_algebra.py},\\
 &\texttt{verification/verify\_n\_channel.sage}.
 \end{split}
\end{equation}
The first uses SageMath \cite{SageMath}.  It constructs
\begin{equation}\label{eq:app-verification-raw-matrix}
 u^nI_2-E\diag(1,-1)
 -\lambda u^r\begin{pmatrix}a&b\\-b&d\end{pmatrix}
\end{equation}
and asks Sage to form its determinant and discriminant.  Only afterward are
these objects compared with the displayed formulas.  The SymPy script
rebuilds the same matrix independently and checks the exact trinomial
discriminant, cyclic coupling character, boundary charts and their local
$A$-type models, and quartic invariants.
The two two-channel implementations do not share generated algebraic
expressions.  The third script starts again from raw $N\times N$ matrices.
It checks the classical generic genus and the indefinite two-energy
specialization through $N=5$, including negative numerators, normalization
genera, irreducibility, Newton polygons, and representative torus
nondegeneracy systems.  It is a regression test for
\cref{sec:higher-channel-context}, not evidence of novelty.

The exact checks are summarized in \cref{tab:verification}.

\begin{table}[ht]
\centering
\caption{Definition-level verification.  All entries are exact unless noted.}
\label{tab:verification}
\begin{tabular}{@{}p{0.31\textwidth}p{0.57\textwidth}@{}}
\toprule
Object & Independent check\\
\midrule
Raw determinant & Direct $2\times2$ determinant for $n=1,\ldots,6$ and
positive and negative reduced numerators\\
Birational model & Discriminant extracted from the raw quadratic and
complete-the-square identity checked coefficientwise\\
Coupling decoration & Cyclic character and the nonzero local units determining
the coupling-function orders at zero and infinity\\
Generic genus & Exact branch counts and Sage genera $0,1,\ldots,5$ for one
fixed generic matrix; separate infinity chart\\
Degenerate loci & Degree drop, diagonal splitting, $\Delta_V=0$, $C=0$, and
odd-multiplicity square-class reduction\\
Higher-channel context & Raw determinants through $N=5$; generic genera and
the $H_0=EJ$ values in Proposition~\ref{prop:higher-channel-EJ}\\
Boundary models & Both $\lambda$ and $1/\lambda$ chart discriminants and
completed-square $A_{2|r|-1}$ models for $|r|=1,\ldots,4$\\
Elliptic moduli & Binary-quartic syzygy, matrix $j$-map, branch cross-ratio,
and both directions of the canonical birational map\\
Interpolation & General master quadratic and exact raw-resolvent examples in
quadratic number fields\\
\bottomrule
\end{tabular}
\end{table}

The canonical elliptic model is also constructed as a Sage elliptic curve.
Sage returns $j=1728$, discriminant $-32768$, conductor $256$, and Cremona
label \texttt{256b2}; the last two items are cross-checked against the LMFDB
record.  A 200-bit branch cross-ratio evaluation is included as an auxiliary
check of the exact rational $j$ formula.  It is the only floating-point
assertion and is not used to establish the theorem.

\subsection{Reproduction Commands}

From the manuscript directory, the complete computational check is
\begin{verbatim}
sage verification/verify_full.sage
python3 verification/verify_algebra.py
sage verification/verify_n_channel.sage
python3 figures/generate_figures.py
\end{verbatim}
The final verified environment used SageMath 10.9, Python 3.9.6, SymPy 1.14.0,
NumPy 1.26.4, and Matplotlib 3.9.2.  No network access or downloaded data are
needed.  Successful execution prints separate \texttt{PASS} lines for the
raw determinants, branch geometry, degenerations, elliptic maps,
coupling decoration, interpolation, and figure data.

\subsection{Figure Provenance}

All figures are generated by
\texttt{figures/generate\_figures.py}.  Vector PDF files are used in the
manuscript, with high-resolution PNG copies retained for review.  Figure~1
uses only $z=u^2$ and $v^2=2-u^4$.  Figure~2 reconstructs the roots of
\begin{equation}\label{eq:app-figure-Q}
 Q_n=-3u^{2n}-10u^n+29,
 \qquad n=1,2,3,4,
\end{equation}
and asserts that exactly $2n$ distinct branch points were found.  Figure~3
enumerates reduced rational exponents with denominator at most $24$ and plots
$g=n-1$.  It also displays a rational sequence converging to $1/2$ whose
denominators, and hence generic genera, diverge.  The plots are expository
representations of exact data, not numerical evidence for the theorems.

The generator uses a colorblind-safe palette together with redundant marker
and line styles, so branch orbits remain distinguishable in grayscale.  The
source data, generation script, and publication PDFs are kept together to
make each graphical claim reproducible.

\section*{Acknowledgments}
This work was supported by the National High-Level Overseas Talent Program (KS21400126), the Suzhou Talent project (ZXP2025057), the Jiangsu Distinguished Professorship Fund (SR21400225), and the Research Start-up Fund (NH21400525).

\bibliographystyle{amsplain}
\bibliography{references}

\providecommand{\bysame}{\leavevmode\hbox to3em{\hrulefill}\thinspace}
\providecommand{\MR}{\relax\ifhmode\unskip\space\fi MR }
\providecommand{\MRhref}[2]{%
  \href{http://www.ams.org/mathscinet-getitem?mr=#1}{#2}
}
\providecommand{\href}[2]{#2}
\begin{thebibliography}{10}

\bibitem{ChanCorless2016}
Eunice Y.~S. Chan and Robert~M. Corless, \emph{{Narayana}, {Mandelbrot}, and a
  new kind of companion matrix}, arXiv:1606.09132, 2016.

\bibitem{DussonSigalStamm2021}
Genevi{\`e}ve Dusson, Israel~Michael Sigal, and Benjamin Stamm, \emph{The
  {Feshbach--Schur} map and perturbation theory}, Partial Differential
  Equations, Spectral Theory, and Mathematical Physics, EMS Series of Congress
  Reports, EMS Press, 2021, pp.~65--88.

\bibitem{FaulknerEtAl2011}
Thomas Faulkner, Hong Liu, John McGreevy, and David Vegh, \emph{Emergent
  quantum criticality, {Fermi} surfaces, and {AdS}$_2$}, Physical Review D
  \textbf{83} (2011), no.~12, 125002.

\bibitem{Fisher2008}
Tom~A. Fisher, \emph{The invariants of a genus one curve}, Proceedings of the
  London Mathematical Society \textbf{97} (2008), no.~3, 753--782.

\bibitem{GutierrezShaska2005}
J.~Gutierrez and T.~Shaska, \emph{Hyperelliptic curves with extra involutions},
  LMS Journal of Computation and Mathematics \textbf{8} (2005), 102--115.

\bibitem{GuttelTisseur2017}
Stefan G{\"u}ttel and Fran{\c{c}}oise Tisseur, \emph{The nonlinear eigenvalue
  problem}, Acta Numerica \textbf{26} (2017), 1--94.

\bibitem{Izosimov2015}
Anton Izosimov, \emph{Matrix polynomials, generalized {Jacobians}, and
  graphical zonotopes}, arXiv:1506.05179, 2015.

\bibitem{Kochubei2011}
Anatoly~N. Kochubei, \emph{General fractional calculus, evolution equations,
  and renewal processes}, Integral Equations and Operator Theory \textbf{71}
  (2011), no.~4, 583--600.

\bibitem{KollarMiller2014}
Richard Koll{\'a}r and Peter~D. Miller, \emph{Graphical {Krein} signature
  theory and {Evans--Krein} functions}, SIAM Review \textbf{56} (2014), no.~1,
  73--123.

\bibitem{Liu2026LocalHP}
Kejun Liu, \emph{A local {Hilbert--P{\'o}lya} realisation for elliptic curve
  {$L$}-functions}, arXiv:2605.17645v2, 2026.

\bibitem{MaldacenaStanford2016}
Juan Maldacena and Douglas Stanford, \emph{Remarks on the {Sachdev--Ye--Kitaev}
  model}, Physical Review D \textbf{94} (2016), no.~10, 106002.

\bibitem{MuellerPink2022}
Nicolas M{\"u}ller and Richard Pink, \emph{Hyperelliptic curves with many
  automorphisms}, International Journal of Number Theory \textbf{18} (2022),
  no.~4, 913--930.

\bibitem{Shaska2003}
Tanush Shaska, \emph{Determining the automorphism group of a hyperelliptic
  curve}, Proceedings of the 2003 International Symposium on Symbolic and
  Algebraic Computation, ACM, 2003, pp.~248--254.

\bibitem{Silverman2009}
Joseph~H. Silverman, \emph{The arithmetic of elliptic curves}, 2 ed., Graduate
  Texts in Mathematics, no. 106, Springer, New York, 2009.

\bibitem{LMFDB256b2}
{The LMFDB Collaboration}, \emph{Elliptic curve {256.b2}},
  \url{https://www.lmfdb.org/EllipticCurve/Q/256/b/2}, 2026, Accessed July 30,
  2026.

\bibitem{SageMath}
{The Sage Developers}, \emph{{SageMath}, the sage mathematics software system},
  2026, Version 10.9.

\bibitem{Vivolo2000}
Olivier Vivolo, \emph{Jacobians of singular spectral curves and completely
  integrable systems}, Proceedings of the Edinburgh Mathematical Society
  \textbf{43} (2000), 605--623.

\end{thebibliography}

\end{document}